\documentclass{amsart}

\usepackage[all]{xy}
\usepackage{amsmath}
\usepackage{hyperref}
\usepackage{cleveref}
\usepackage{amsfonts,graphics,amsthm,amsfonts,amscd,latexsym}
\usepackage{epsfig}
\usepackage{flafter}
\usepackage{mathtools}
\usepackage{comment}
\usepackage{stmaryrd}
\usepackage{tabularx}
\usepackage{ytableau}
\usepackage{subcaption}
\usepackage{pst-node}
\usepackage{auto-pst-pdf}
\usepackage{tikz-cd}
\hypersetup{
    colorlinks=true,    
    linkcolor=blue,          
    citecolor=blue,      
    filecolor=blue,      
    urlcolor=blue           
}
\usepackage{pgfplots}
\pgfplotsset{compat=1.15}
\usepackage{mathrsfs}
\usepackage{tikz}
\usetikzlibrary{graphs,positioning,arrows,shapes.misc,decorations.pathmorphing}

\tikzset{
    >=stealth,
    every picture/.style={thick},
    graphs/every graph/.style={empty nodes},
}

\tikzstyle{vertex}=[
    draw,
    circle,
    fill=black,
    inner sep=1pt,
    minimum width=5pt,
]
\usepackage{amssymb}
\usepackage{color}

\calclayout

\usetikzlibrary{decorations.pathmorphing}
\tikzstyle{printersafe}=[decoration={snake,amplitude=0pt}]

\newcommand{\Aut}{\operatorname{Aut}}

\newcommand{\Cl}{\operatorname{Cl}}

\newcommand{\rank}{\operatorname{rank}}

\newcommand{\Hom}{\operatorname{Hom}}

\newcommand{\Bir}{\operatorname{Bir}}

\newcommand{\PGL}{\mathbb{P}{\rm GL}}
\newcommand{\Cr}{\operatorname{Cr}}
\newcommand{\coreg}{\operatorname{coreg}}

\newcommand{\pp}{\mathbb{P}}

\newcommand{\qq}{\mathbb{Q}}
\newcommand{\zz}{\mathbb{Z}}

\newcommand{\cc}{\mathbb{C}}

\DeclareMathOperator{\mld}{mld}

\usepackage{tikz}
\usetikzlibrary{matrix,arrows,decorations.pathmorphing}
\usetikzlibrary{arrows}

  \newtheorem{introthm}{Theorem}

  \newtheorem{theorem}{Theorem}[section]
  \newtheorem{lemma}[theorem]{Lemma}
  \newtheorem{proposition}[theorem]{Proposition}

  \newtheorem{notation}[theorem]{Notation}
  \newtheorem{definition}[theorem]{Definition}
  \newtheorem{example}[theorem]{Example}

  \newtheorem{question}[theorem]{Question}

\theoremstyle{definition}

\newtheorem{remark}[theorem]{Remark}

\usepackage[margin=1in]{geometry} 
\usepackage{amsmath,amssymb,mathtools, mathrsfs,esint,tikz-cd,verbatim}
\newcommand{\R}{{\mathbb R}}

\newcommand{\Z}{{\mathbb Z}}
\newcommand{\Q}{{\mathbb Q}}

\newcommand{\C}{{\mathbb C}}

\newcommand{\on}[1]{\operatorname{#1}}

\newcommand{\GL}{{\on{GL}}}

\newcommand{\SO}{{\on{SO}}}

\newcommand{\defi}[1]{\textsf{#1}} 

\theoremstyle{remark}

\numberwithin{equation}{section}

\keywords{Fano varieties, symmetric groups, weighted complete intersections, toric varieties}

\subjclass[2020]{Primary: 14J45, 14J50; Secondary: 14E07, 14M10}

\thanks{During the preparation of this article, L.E. was partially supported by NSF grant DMS-2054553, and L.J. was partially supported by NSF MSPRF grant DMS-2202444.}

\begin{document}

\title[Symmetries of Fano varieties]{Symmetries of Fano varieties}

\author[L.~Esser]{Louis Esser}
\address{Department of Mathematics, Princeton University, Fine Hall, Washington Road, Princeton, NJ 08544-1000, USA
}
\email{esserl@math.princeton.edu}

\author[L.~Ji]{Lena Ji}
\address{Department of Mathematics, University of Michigan, 530 Church Street, Ann Arbor, MI 48109-1043}
\email{lenaji.math@gmail.com}

\author[J.~Moraga]{Joaqu\'in Moraga}
\address{UCLA Mathematics Department, Box 951555, Los Angeles, CA 90095-1555, USA
}
\email{jmoraga@math.ucla.edu}

	\maketitle
	
\begin{abstract}
We study Fano varieties endowed with a faithful action of a symmetric group, as well as analogous results for Calabi--Yau varieties, and log terminal singularities.
We show the existence of a constant $m(n)$, 
so that every symmetric group $S_k$ acting on an $n$-dimensional Fano variety satisfies $k \leq m(n)$.
We prove that $m(n)> n+\sqrt{2n}$ for every $n$.
On the other hand, we show that $\lim_{n\to \infty} m(n)/(n+1)^2 \leq 1$. However, this asymptotic upper bound is not expected to be sharp.

We obtain sharp bounds for certain classes of varieties. For toric varieties, we show that \(m(n)=n+2\) for \(n\geq 4\). For Fano quasismooth weighted complete intersections, we prove the asymptotic equality $\lim_{n\to \infty} m(n)/(n+1)=1$. Among the Fano weighted complete intersections, we study the maximally symmetric ones and show that they are closely related
to the {\em Fano--Fermat varieties}, i.e.,
Fano complete intersections in \(\pp^N\)
cut out by Fermat hypersurfaces.

Finally, we draw a connection between maximally symmetric Fano varieties 
and boundedness of Fano varieties.
For instance, we show that the class of $S_8$-equivariant Fano $4$-folds forms a bounded family. In contrast, the \(S_7\)-equivariant Fano \(4\)-folds are unbounded.
\end{abstract}

\setcounter{tocdepth}{1}
\tableofcontents

\section{Introduction}

In this paper, we study automorphisms of Fano varieties. The automorphism group of a Fano variety satisfies the so-called {\em Jordan property}.
This property states that any finite subgroup of the automorphism group contains a normal abelian subgroup of bounded index. Moreover, for \(n\)-dimensional Fano varieties, there is a uniform upper bound for this index that only depends on \(n\) (\cite[Theorem 1.8]{PS16} and \cite[Corollary 1.5]{Bir21}).
In particular, if a finite semi-simple group
acts on a $n$-dimensional Fano variety,
its order is bounded above in terms of $n$.
The symmetric groups $S_n$, for $n\geq 5$, are very natural examples of semi-simple groups.

In dimension \(1\), any symmetric group
acts on some curve of general type; however,
the symmetric actions
on elliptic 
and rational curves are much more limited. 
For instance, 
$S_4$ is the largest symmetric group acting on a rational curve. 
In higher dimensions, the Jordan property gives a (non-explicit) upper bound for the order of symmetric groups acting on $n$-dimensional Fano varieties.
Similar behavior is expected in the case
of Calabi--Yau varieties (see, e.g.,~\cite[Conjecture 4.46]{Mor22}).
However, in neither of these cases do we understand
how to control the size of the symmetric group
in terms of the dimension of the variety
endowed with the action.

As a first na\"{i}ve example, one can consider symmetric actions on projective spaces $\pp^n$.
Although $S_4$ acts on $\pp^1$ and $S_6$ acts on $\pp^3$, in most dimensions the largest symmetric group that acts faithfully on $\pp^n$ is $S_{n+2}$, via the standard representation of that group in $\GL_{n+1}(\C)$.  In fact, $S_{n+2}$ is the largest symmetric group inside $\Aut(\pp^n) = \PGL_{n+1}(\C)$ for $n = 2$ and $n \geq 4$ (cf. \Cref{tab:S_k_reps}).
However, this example is not optimal, even among rational varieties:
in each dimension $n$ there exists a smooth rational Fano variety that admits an $S_{n+3}$-action (\Cref{ex:rat-n+3}).
Moreover, there exist (conjecturally irrational) Fano varieties with even larger symmetric actions (see \Cref{sec:ex}).
On the other hand, based on work by J. Xu~\cite{Xu-p-groups}, we give a first asymptotic upper bound
for the order of symmetric groups
acting on $n$-dimensional Fano varieties (Fano varieties in this paper
have klt singularities, by definition).

\begin{introthm}\label{introthm:quadratic-bound-Fano}
Let $S_{m(n)}$ be the largest symmetric
group acting faithfully on an $n$-dimensional Fano variety.
Then, we have that 
\[
\lim_{n\to \infty} \frac{m(n)}{(n+1)^2} \leq 1.
\]
\end{introthm}

By means of global-to-local techniques, we show that the previous statement admits an analog for klt singularities. 

\begin{introthm}\label{introthm:quadratic-bound-klt}
Let $S_{\ell(n)}$ be the largest symmetric group acting faithfully on an $n$-dimensional klt singularity.
Then, we have that 
\[
\lim_{n\to \infty} \frac{\ell(n)}{n^2}\leq 1.
\]
\end{introthm}

We emphasize that \Cref{introthm:quadratic-bound-Fano} is also expected to hold for Calabi--Yau varieties. However, \Cref{introthm:quadratic-bound-klt} does not hold for log canonical singularities.
Indeed, every symmetric group acts on some $3$-dimensional log canonical singularity (see, e.g.,~\cite[Theorem 6]{FM23}).
Neither \Cref{introthm:quadratic-bound-Fano}
nor \Cref{introthm:quadratic-bound-klt}
is expected to be sharp.

\subsection{Weighted complete intersections and toric varieties}
For more restrictive classes of varieties,
we can prove sharp bounds on symmetric
actions in every dimension.
First, we find the largest symmetric action 
on a simplicial toric variety (not 
necessarily Fano) in every dimension.

\begin{introthm}[c.f.~\Cref{thm:toric}]\label{introthm:toric}
Let \(X\) be a complete simplicial toric variety of dimension \(n\). Suppose that the symmetric group \(S_k\) acts faithfully on \(X\).  If $n = 1,2,$ or $3$, then $k \leq n+3$; if $n \geq 4$, then $k \leq n+2$.

These bounds are sharp for each $n$.  If equality is achieved and $n \neq 2,4$, then $X \cong \pp^n$. If $n = 2$, then $k = 5$ if and only if $X \cong \pp^1 \times \pp^1$.  If $n = 4$, then $k = 6$ if and only if $X \cong \pp^4$ or $X \cong \pp^2 \times \pp^2$.
\end{introthm}

For $n \geq 3$, however, toric 
varieties do not have the largest 
symmetric actions among all Fano varieties.
Rather, we expect the optimal examples
to be (quasismooth) weighted complete intersections.
In Sections \ref{sect:wci} and \ref{sect:max_sym_vars},
we prove sharp bounds on symmetric actions
on these varieties in every dimension.
In particular, we prove the following result.

\begin{introthm}[c.f.~\Cref{thm:max_symmetric_action}]\label{introthm-bound-Fano-WCI}
Let $S_{m(n)}$ be the largest symmetric group
acting faithfully on an $n$-dimensional quasismooth Fano weighted complete intersection.
Then, we have that 
\[
\lim_{n\to \infty}\frac{m(n)}{n+1}=1. 
\]
\end{introthm}

Though the asymptotics are the same as in the toric case, $m(n)$ is on the order of $n + \sqrt{2n}$ here. We obtain a precise formula for this $m(n)$ in \Cref{thm:max_symmetric_action}, where we also prove an analogous statement for Calabi--Yau weighted complete intersections.
We expect a similar statement to~\Cref{introthm-bound-Fano-WCI} in the case of weighted complete intersection klt singularities (see~\Cref{quest:klt-WCI}).

An $n$-dimensional Fano variety
is said to be \defi{maximally symmetric}
if it admits the largest symmetric action among $n$-dimensional Fano varieties.
We define 
\defi{maximally symmetric} (quasismooth) Fano weighted complete intersections
similarly. 
Our next aim is to describe 
maximally symmetric
Fano
weighted complete intersections.

The following theorem
gives a characterization of maximally symmetric Fano complete intersections.
In the theorem below, we say that a complete intersection in $\pp^N$ is \defi{totally symmetric} if its defining ideal is contained in the ring of symmetric polynomials in the variables $x_0,\ldots,x_N$.

\begin{introthm}\label{introthm:maximally-symmetric-Fano-WCI}
Let $X$ be a maximally symmetric quasismooth Fano weighted complete intersection of dimension $n \neq 2$ where the maximal action is by $S_k$. Then there is a finite cover $X \rightarrow Y$, where $Y$ is a totally symmetric complete intersection in $\pp^{k-1}$.
\end{introthm}

In the theorem below, the \defi{index} of $-K_X$ refers to the largest positive integer $r$ such that $-K_X$ is divisible by $r$ in the class group $\Cl X$.
A \defi{Fano--Fermat variety} is a Fano complete intersection in the projective space $\pp^N$ that is cut out by Fermat hypersurfaces. 

\begin{introthm}[c.f.~\Cref{thm:maximally_symmetric_largest_index}]
\label{introthm:maximally-symmetric-large-index-Fano-WCI}
Let $X$ be a maximally symmetric quasismooth Fano weighted complete intersection of dimension $n$ with largest possible index of $-K_X$, where the maximal action is by $S_k$.  Then $X$ is $S_k$-equivariantly isomorphic 
to a Fano--Fermat variety.
\end{introthm}

More precisely, we will show that a 
maximally symmetric Fano weighted complete intersection
with maximal index
is isomorphic
to the Fano--Fermat variety in \Cref{ex:optimal_example}.

\subsection{Symmetries and boundedness}
Boundedness of Fano varieties
is an important topic in birational geometry.
Koll\'ar, Miyaoka, and Mori proved the boundedness
of $n$-dimensional smooth Fano varieties~\cite{KMM92}.
Birkar proved the boundedness of $n$-dimensional 
Fano varieties with mininimal log discrepancy bounded away from zero~\cite{Bir21}.
Other constraints on invariants are also known to
give boundedness of $n$-dimensional Fano varieties,
such as bounding the degree and alpha-invariant away from zero~\cite{Jia20}.
In these cases, 
the invariant that defines a bounded family of Fano varieties is a measure of singularities.
We prove a boundedness result in a novel direction---we show that Fano $4$-folds with large symmetric actions form bounded families:

\begin{introthm}[c.f.~\Cref{thm:S8-4-log-pair-bounded}]\label{introthm:S8-4-fold-bounded}
The class of $S_8$-equivariant klt Fano $4$-folds is bounded.
\end{introthm}

In contrast, the \(S_7\)-equivariant klt Fano \(4\)-folds are unbounded (see \Cref{ex:S_7-Fano-4-fold}).

It is worth mentioning that we are not aware of moduli in the bounded family from \Cref{introthm:S8-4-fold-bounded}.
This means that all the examples that we know are isolated (see \Cref{ex:max-sym-Fano-4-fold} and \Cref{quest:moduli}).
Note that for \(n\leq 3\), the classification shows that there are only finitely many maximally symmetric Fanos of dimension \(n\) \cite{DI09,Prokhorov_space_Cremona}; in fact, for \(n=3\) there is only one up to conjugation.

The proof of \Cref{introthm:S8-4-fold-bounded} uses several results in geometry: 
finite actions on spheres~\cite{zim18}, 
finite actions on rationally connected varieties~\cite{BCDP-23},
dual complexes of log Calabi--Yau pairs~\cite{KX16}, 
and boundedness of Fano varieties~\cite{Bir21}.
We expect that maximally symmetric $n$-dimensional Fano varieties
form bounded families (see \Cref{quest:max-symm-bound}). 
Let us emphasize that the behavior 
described
in \Cref{introthm:S8-4-fold-bounded} is not expected
for actions by other finite groups. 
For instance, 
every $n$-dimensional toric Fano variety 
admits the action of $(\zz/m)^n$ for $m$ arbitrarily large.
However, in the case of finite abelian actions, 
we have some structural theorems instead.
In~\cite[Theorem 2]{Mor20c}, it is proved that $n$-dimensional Fano varieties with $(\zz/m)^n$-actions for $m$ large are compactifications of $\mathbb{G}_m^n$.

In a similar vein, we prove a local statement for $5$-dimensional klt singularities with faithful $S_8$-actions. 
In this case, as it is usual in the domain of singularities, 
we only get a bounded family up to degeneration (see \Cref{def:bounded-deg}).

\begin{introthm}\label{introthm:S8-5-dim-klt}
Let $\epsilon>0$.
The class of $S_8$-equivariant $5$-dimensional klt singularities with minimal log discrepancy at least $\epsilon$
forms a family which is bounded up to degeneration.
\end{introthm}

We summarize the largest known maximal symmetric actions on various types of varieties and topological spaces in \Cref{tab:summary}.

\begin{table}[h]
    \renewcommand{\arraystretch}{1.2}
    \centering
    \begin{tabular}{c|c|c|c|c}
   \renewcommand{\arraystretch}{1.2}
    Dimension & Fano & Calabi--Yau & Rational & Sphere\\
    \hline 
    1 & 4* & 3* & 4* & 3* \\
    \hline
    2 & 5* & 6* & 5* & 4*\\
    \hline
    3 & 7* & 7 & 6* & 5*\\
    \hline
    4 & 8 & 8 & 7 & 6*\\
    \hline
    5 & 9 & 10 & 8 & 7  \\
    \hline
    $n\gg 0$ & $n + \sqrt{2n} + O(1)$ &  $n + \sqrt{2n} + O(1)^{\dag}$ & $n+3$ & $n+2$\\
\end{tabular} 
    \caption{Each table entry shows the maximal $k$ for which $S_k$ is known to act faithfully on an object of the indicated class and dimension, to our knowledge.  In the case of the $n$-sphere, we consider topological actions.  Entries with * are known to be optimal.  The expressions for Fano and Calabi--Yau $n$-folds are approximate; precise formulas appear in \Cref{sect:wci}.  In the asymptotic Calabi--Yau case $(\dag)$, examples with these asymptotics are only known for infinitely many values of $n$, rather than all $n$.  See \Cref{rem:CY_sharp}.}
    \label{tab:summary}
\end{table}

The fact that $S_4$ is the largest symmetric group acting on $\pp^1$ is a classical result (see, e.g.,~\cite{beauville-PGL2}).
For the action of $S_5$ on Fano surfaces, we refer
to the work of Dolgachev and Iskovskikh~\cite{DI09}.
The fact that $S_3$ is the largest symmetric group acting on curve of genus $1$ is classical (see, e.g.,~\cite{Sil09}).
The fact that $S_6$ is the largest symmetric group acting on Calabi--Yau surfaces follows 
from the work of Mukai and Fujiki~\cite{Muk88,Fuj88}.
For the actions of symmetric groups on rationally connected varieties of dimension at most $3$,
we refer the reader to the work of Blanc, Cheltsov, Duncan, and Prokhorov~\cite{BCDP-23,Pro12,Prokhorov_space_Cremona}.
For smooth actions of symmetric groups on spheres of dimension at most $4$, we refer the reader to the work of 
Mecchia and Zimmerman~\cite{zim18,MZ06}.  The symmetric group $S_{n+2}$ acts on the $n$-sphere for any $n \geq 1$ as the symmetries of the (boundary of the) regular $(n+1)$-simplex.  The examples for the remaining table entries appear in \Cref{sec:ex}.

\subsection{Outline}
We begin with preliminary results in \Cref{sec:preliminaries}.
In \Cref{sec:fano-bound}, we prove the quadratic bounds for Fano varieties and klt singularities in Theorems~\ref{introthm:quadratic-bound-Fano} and \ref{introthm:quadratic-bound-klt}. Next, we study toric varieties and weighted complete intersections. We prove \Cref{introthm:toric} on toric varieties in \Cref{sec:toric}. In Sections~\ref{sect:wci} and~\ref{sect:max_sym_vars}, we study weighted complete intersections. In \Cref{sect:wci}, we consider Fano and Calabi--Yau weighted complete intersections: we show the implications of a large symmetric action on the defining equations of such a weighted complete intersection, and we prove \Cref{introthm-bound-Fano-WCI}. In \Cref{sect:max_sym_vars}, we study the maximally symmetric Fano case and prove Theorems~\ref{introthm:maximally-symmetric-Fano-WCI} and~\ref{introthm:maximally-symmetric-large-index-Fano-WCI}.
Next, in \Cref{sec:sym-bound}, we prove the boundedness results of Theorems~\ref{introthm:S8-4-fold-bounded} and~\ref{introthm:S8-5-dim-klt}. Finally, in \Cref{sec:ex}, we end the article with several examples and questions.

\subsection*{Notation}
We work over the field of complex numbers $\cc$.
Throughout the article
$\zz/m$ denotes the cyclic group with $m$ elements.
\(S_k\) and \(A_k\) denote the symmetric and alternating groups, respectively, on a set of order \(k\). For a finite set \(W\), we also use \(S_W\) to denote the symmetric group on \(W\).

Let \(X\) be a variety. Its automorphism group (regarded with the reduced scheme structure) will be denoted \(\Aut(X)\). For a subscheme \(Z\subset X\), we let \(\Aut(X,Z) \leqslant\Aut(X)\) denote the subgroup of automorphisms that fix \(Z\) (not necessarily pointwise). The Weil divisor class group of a normal variety \(X\) is denoted \(\Cl X\).

\subsection*{Acknowledgements}
The authors would like to thank Serge Cantat, 
Mattia Mecchia,
Yuri Prokhorov, Zinovy Reichstein, and Burt Totaro
for many useful comments.
Part of this work was carried out during a visit of L.J. to the University of California, Los Angeles
and a visit of J.M. to the University of Michigan.
The authors would like to thank these institutions, as well as Burt Totaro and Mircea Musta\c{t}\u{a}, for their hospitality and nice working environment.

\section{Preliminaries}\label{sec:preliminaries}

In this section, we recall some preliminaries regarding
representation theory of symmetric groups, 
singularities of the MMP, 
Fano and Calabi--Yau varieties, 
and boundedness of varieties.

\subsection{Representation theory of symmetric and alternating groups}\label{sec:rep_S_k}

In this section, we review the linear and projective representation theory of alternating and symmetric groups that are required for our proofs. For the projective representation theory of $A_k$ and $S_k$, we refer to \cite{Stembridge,WalesSn}.

\begin{definition}[{\cite[Section 6.9]{Weibel-homological}}]
    {\em 
    A \defi{central extension} of a group \(G\) is an extension \(1\to K\to H\to G\to 1\) such that \(K\) is in the center of \(H\). \(H\) is called a \defi{universal central extension} of \(G\) if for every central extension \(1\to K'\to H'\to G\to 1\) there is a unique homomorphism from \(H\) to \(H'\) over \(G\): \[\xymatrix{1 \ar[r] & K \ar[r] \ar[d] & H \ar[r] \ar[d]^-{\exists!} & G \ar[r] \ar@{=}[d] & 1 \\ 1 \ar[r] & K' \ar[r] & H' \ar[r] & G \ar[r] & 1.}\] If a universal central extension of \(G\) exists, then it is unique up to isomorphism over \(G\).  A group $G$ has a universal central extension if and only if it is perfect \cite[Theorem 6.9.5]{Weibel-homological}.
    }
\end{definition}

Central extensions of finite groups $G$ are important for classifying their projective representations, i.e., embeddings $G \hookrightarrow \Aut(\pp^n) = \PGL_{n+1}(\C)$. Indeed, any projective representation of $G$ in $\PGL_{n+1}(\C)$ gives rise to a linear representation of a central extension in $\GL_{n+1}(\C)$, whose projectivization is the original representation.  For the alternating and symmetric groups, one can use a single central extension to classify all projective representations.  This classification was first achieved by Schur \cite{Schur}.

\begin{example}[{\cite[Example 6.9.10]
{Weibel-homological}}]\label{exmp:central-extn-A_k}
    {\em 
    The regular representation \(A_k\to \SO_{k-1}\) of the alternating group gives rise to a central extension \[\xymatrix{ 1 \ar[r] & \zz/2 \ar[r] & 2\cdot A_k \ar[r] & A_k \ar[r] & 1 }\] by restricting the central extension \(1\to\zz/2\to\mathrm{Spin}_{k-1}(\mathbb R)\to \SO_{k-1} \to 1\). For \(k\geq 5\), the group \(A_k\) is perfect, and the universal central extension is the \defi{Schur covering group}, which we denote \(\tilde{A}_k\). The Schur multiplier is \[ H^2(A_k,\C^*) = \begin{cases} 0 &\quad k \leq 3, \\ \zz/2 &\quad k\in\{4,5\} \cup \zz_{\geq 8}, \\ \zz/6 &\quad k\in\{6,7\}. \end{cases}\] For \(k=5\) and \(k\geq 8\), \(\tilde{A}_k\) is the double cover \(2\cdot A_k\).

    For \(k=6, 7\), there are additional covers \(3\cdot A_k\) and \(6\cdot A_k\) (which are central extensions of \(A_k\) by \(\zz/3\) and \(\zz/6\), respectively), and we have \(\tilde{A}_k=6\cdot A_k\). See \cite[Sections 2.7.3 and 2.7.4]{Wilson-groups} for the constructions of the triple covers \(3\cdot A_k\).
    }
\end{example}

\begin{example}\label{exmp:central-extn-S_k}
{\em 
    The Schur multiplier of $S_k$ is given by
    \[ H^2(S_k,\C^*) = \begin{cases} 0 &\quad k \leq 3, \\ \zz/2 &\quad k \geq 4. \end{cases}\]
    Unlike in the case of $A_k$, $S_k$ is not a perfect group and there is no universal central extension.  In fact, the regular representation $S_k \rightarrow \mathrm{O}(k-1)$ gives rise to two possible central extensions $2 \cdot S^{\pm}_k$ of order $2$ when $k \geq 4$.   These are the restrictions of $\mathrm{Pin^{\pm}}(\R) \rightarrow \mathrm{O}(k-1)$, where $\mathrm{Pin^{\pm}}(\R)$ is one of the two pin groups.  Both extensions are maximal, but they are only isomorphic when $k = 6$.  
    
    However, the representation theory of $S^+_k$ is essentially the same as that of $S^-_k$.  In particular, the dimensions of their irreducible representations are the same \cite[page 93]{Stembridge}.  From now on, we'll denote by $\tilde{S}_k$ a Schur cover of $S_k$, and not make a distinction between the two possible choices for $k \geq 4$.
}
\end{example}

The automorphism groups of varieties appearing in this paper are often the quotients of linear algebraic groups by central subgroups.  For example, $\Aut(\pp^n) \cong \PGL_{n+1}(\C) = \GL_{n+1}(\C)/\C^*$, where $\C^*$ is the group of scalar matrices.  To identify embeddings of $S_k$ or $A_k$ in these automorphism groups, we therefore need to understand the representation theory of their central extensions.  A faithful linear representation of $S_k$ or $\tilde{S}_k$ of dimension $n+1$, for instance, gives a projective representation of $S_k$ in dimension $n$.  \Cref{tab:S_k_reps} lists the smallest degrees of faithful representations of $S_k$ and $\tilde{S}_k$ for all $k \geq 4$.  The smallest value in each row determines the largest symmetric group acting on $\pp^n$, namely $S_4$ if $n = 1$, $S_6$ if $n = 3$, and $S_{n+2}$ for all other values of $n$.

\begin{table}
    \centering
    \begin{tabular}{c|c|c}
    \(k\) & \(S_k\) & \(\tilde{S}_k\) \\
    \hline
    4 & 3 & 2 \\
    \hline 
    5 & 4 & 4 \\
    \hline
    6 & 5 & 4 \\
    \hline
    \(\geq 7\) & \(k-1\) &  \(2^{\lfloor (k-1)/2 \rfloor}\)\\
\end{tabular}
\caption{The table above summarizes the representation theory of $S_k$ and $\tilde{S}_k$ for $k \geq 4$. The second column shows the degree of the smallest faithful representation of $S_k$, and the third of $\tilde{S}_k$.}
    \label{tab:S_k_reps}
\end{table}

In \Cref{sect:wci}, we'll require the following further lemma about symmetric group representations.  It is expressed in terms of the function

$$c_{\mathrm{Fano}}(n) \coloneqq n + \left\lceil \frac{1 + \sqrt{8n + 9}}{2} \right\rceil,$$
which will be important in \Cref{sect:wci}.

\begin{lemma}
\label{sym_group_reps}
Let $n \geq 4$ and $k \geq c_{\mathrm{Fano}}(n)$.  Let \(S_k\) be the symmetric group of order \(k\), and let $\tilde{S}_k$ be a representation group of $S_k$.  Unless $n = 4$ and $k = c_{\mathrm{Fano}}(4) = 8$, the only irreducible representations of $\tilde{S}_k$ with dimension at most $2n+2$ are: the trivial representation, the sign representation (of dimension $1$), the standard representation of $S_k$ (of dimension $k - 1$), and the tensor product of the standard and sign representations (of dimension $k-1$).

In the special case of $n = 4$, $k = 8$, there is also a faithful representation of $\tilde{S}_8$ of dimension $8$.  Any polynomial in $8$ variables which is $\tilde{S}_8$-invariant up to sign for this representation is contained in the invariant ring $\C[z_1,\ldots,z_8]^{\tilde{A}_8}$ given by restriction of the basic spin representation to the subgroup $\tilde{A}_8$.  This ring has lowest degree generators $h_1, h_2, h_3$ of degrees $2, 8$, and $8$, respectively.
\end{lemma}

\begin{proof}
The smallest faithful representation of $\tilde{S}_k$ is the \defi{basic spin representation}, which has dimension \(2^{(k-2)/2}\) if \(k\) is even and dimension \(2^{(k-1)/2}\) if \(k\) is odd \cite[Section 3]{Stembridge}.  These dimensions are greater than $2k$ and hence greater than $2n+2$ when $k \geq 11$.  The only remaining cases are $n = 4$ and $n = 5$, where $k \geq c_{\mathrm{Fano}}(4) = 8$ or $k \geq c_{\mathrm{Fano}}(5) = 9$.  Omitting the special case of $n = 4$, $k = 8$, the basic spin representations of $\tilde{S}_k$ have dimension at least $16$ in both cases, which is larger than $2n+2$ for either value of $n$.

If a representation of $\tilde{S}_k$ is not faithful, then it factors through $S_k$.  Therefore, it remains to bound the sizes of representations of $S_k$.  As above, the assumptions of the lemma imply $k \geq 8$.  A result of Rasala \cite[Result 2]{Rasala} gives that the first three dimensions of irreducible representations of $S_k$ are $1, k-1$, and  $\frac{1}{2}k(k-3)$ for $k \geq 9$.  The irreducible representations of dimension $1$ and $k-1$ are precisely those stated in the lemma.  The third value, $\frac{1}{2}k(k-3)$, grows quadratically in $k$ and is greater than $2n+2$ when $k \geq 9$.  In the $k = 8$ case, the next largest representation of $S_8$ has dimension $14$, which is greater than $10 = 2 \cdot 4 + 2$.

Now we return to the case $n = 4$, $k = 8$.  The smallest dimensions of representations of $\tilde{S}_8$ are $1, 7, 8$ (all others have dimension greater than $2n + 2 = 10$).  The $1$- and $7$-dimensional representations are the ones already listed.  The representations of dimension $8$ are the basic spin representations.  A polynomial which is invariant up to sign under the $\tilde{S}_8$-action is in particular an $\tilde{A}_8$-invariant polynomial, where $\tilde{A}_8$ is the (unique) Schur double cover of $A_8$.

The dimensions of the graded pieces of the invariant ring $\C[z_1,\ldots,z_8]^{\tilde{A}_8}$ are readily computable using Molien's formula, for example using \texttt{gap}.  This computation yields that the first few generators have degrees $2, 8$, and $8$.
\end{proof}

Next, we collect results on the minimal degree characters of the alternating group and its Schur cover. We will use these results in \Cref{sec:toric}.
\begin{lemma}\label{lem:irreps-Atilde_k}
    For \(k\geq 4\), let \(\tilde{A}_k\) be the Schur covering group of \(A_k\). The minimal degree faithful representations of \(A_k\) and its central extensions are summarized in~\Cref{tab:A_k_reps}. If \(k=8\) or \(k\geq 10\) then the smallest degree of a nontrivial irreducible representation of \(\tilde{A}_k\) is \(k-1\), and it factors through the standard representation of \(A_k\).
\end{lemma}

    \begin{table}
    \centering
    \begin{tabular}{c|c|c|c|c}
   \renewcommand{\arraystretch}{1.2}
    \(k\) & \(A_k\) & \(2\cdot A_k\) & \(3\cdot A_k\) & \(\tilde{A}_k\) \\
    \hline 
    4 & 3 & 2 & N/A & 2 \\
    \hline 
    5 & 3 & 2 & N/A & 2\\
    \hline
    6 & 5 & 4 & 3 & 6\\
    \hline
    7 & 6 & 4 & 6 & 6 \\
    \hline
    \(\geq 8\) & \(k-1\) & \(2^{\lfloor (k-2)/2 \rfloor} \) & N/A & \(2^{\lfloor (k-2)/2 \rfloor} \) \\
    \end{tabular}
    \caption{The table above summarizes the smallest degrees of the faithful representations of \(A_k\) and its central extensions. The Schur covering group \(\tilde{A}_k\) is $2 \cdot A_k$ for $k = 4,5$, $k \geq 8$ and $6 \cdot A_k $ for $k = 6,7$.}
    \label{tab:A_k_reps}
    \end{table}

\begin{proof}
    The faithful representation of \(2\cdot A_k\) is the basic spin representation, which has degree \(2^{\lfloor \frac{k-2}{2} \rfloor}\) by \cite{WalesSn}. If a representation of \(\tilde{A}_k\) is not faithful and if \(k\neq 6,7\), then it factors through \(A_k\). Every irreducible character of \(A_k\) is obtained from the restriction of an irreducible character of \(S_k\) \cite[Statement 20.13.(3)]{James-Liebeck}, so \Cref{sym_group_reps} implies that \(A_k\) has exactly one nontrivial irreducible character of minimal degree \(k-1\), which is the standard representation. For \(k\geq 8\) we have \(k-1 \geq 2^{\lfloor \frac{k-2}{2} \rfloor}\). Strict inequality holds for \(k=8\) and \(k\geq 10\), so the unique smallest degree representation of \(\tilde{A}_k\) comes from the standard representation of \(A_k\).

    If \(k=6\) or \(7\), then \(2\cdot A_k\) is no longer the Schur cover, and we need to also consider the minimal degree faithful representations of \(3\cdot A_k\) and \(6\cdot A_k\). These can be computed directly using \texttt{gap}.
    
\end{proof}

\subsection{Singularities and positivity of pairs}
In this subsection, we briefly recall some terminology related to singularities of pairs. We refer the reader to~\cite{KM98}.

\begin{definition}
{\em  
A \defi{log pair} $(X,B)$ is a couple 
consisting of a normal quasi-projective variety $X$ and an effective divisor $B$ for which $K_X+B$ is a $\qq$-Cartier divisor.
Let $\pi\colon Y\rightarrow X$ be a projective morphism from a normal variety. 
Let $E\subset Y$ be a prime divisor.
The \defi{log discrepancy} of $(X,B)$ at $E$ is the rational number $1-{\rm coeff}_E(B_Y)$ where $B_Y$ is defined by the formula:
\[
K_Y+B_Y=\pi^*(K_X+B).
\]
We say that $(X,B)$ is \defi{Kawamata log terminal} or \defi{klt} for short if
all the log discrepancies of $(X,B)$ are positive.
We say that $(X,B)$ is \defi{log canonical} or \defi{lc} for short if all the log discrepancies are nonnegative.
We say that $X$ is \defi{klt} (resp. \defi{lc}) if the pair $(X,0)$ is klt (resp. lc).
}
\end{definition}

In~\Cref{sec:sym-bound}, we will consider group actions on log pairs. 

\begin{definition}
{\em 
Let $(X,B)$ be a log pair.
We write $G\leqslant {\rm Aut}(X,B)$
if $G$ is a group acting on $X$
and $g^*B=B$ for every $g\in G$.
In particular, every element of $G$
maps components of $B$ 
to components of $B$ with the same coefficient.

Let $X$ be an algebraic variety, $G\leqslant {\rm Aut}(X)$ be a finite subgroup, and $\pi\colon X\rightarrow Y\coloneqq X/G$ be the quotient.
We say that $\pi$ is {\em quasi-\'etale} if it is \'etale over a big open subset of $Y$.
}
\end{definition}

The main objects of study of this article
are Fano and Calabi--Yau varieties. 

\begin{definition}
{\em 
A \defi{Fano pair} is a log pair $(X,B)$ with klt singularities for which $-(K_X+B)$ is ample. 
If $B=0$, then we simply say that $X$ is a \defi{Fano variety}.
A \defi{Calabi--Yau variety} is a variety $X$ with klt singularities for which $K_X\sim_{\qq}0$.
A \defi{log Calabi--Yau pair} is a log pair $(X,B)$ with log canonical singularities for which $K_X+B\sim_\qq 0$.
}
\end{definition}

Note that we allow log Calabi--Yau pairs to have log canonical singularities.
This is a natural assumption to make when considering boundaries on Fano varieties
that induce a log Calabi--Yau structure.

\subsection{Boundedness of varieties and singularities} 
In this subsection, we recall some concepts about boundedness of varieties and singularities.
In~\Cref{sec:sym-bound}, we will prove some results regarding boundedness
of Fano varieties
and klt singularities
admitting large symmetric actions.

\begin{definition}\label{def:bounded}
{\em 
Let $\mathcal{C}$ be a class of log pairs. 
We say that the class $\mathcal{C}$ is \defi{log bounded}
if the following condition holds.
There exists a finite type morphism
$\mathcal{X}\rightarrow T$ 
and a boundary $\mathcal{B}$ on $\mathcal{X}$
such that every element $(X,B)\in \mathcal{C}$
is isomorphic to 
$(\mathcal{X}_t,\mathcal{B}_t)$ for some closed point $t\in T$.
If we consider a class of varieties instead of pairs,
then we simply say that the class of varieties is \defi{bounded}.
}
\end{definition}

Many classes of varieties or log pairs 
satisfy a boundedness condition when 
certain invariants are fixed.
However, this is not the case for singularities.
Even if we fix many invariants for a class of singularities, it is likely that the resulting class is not bounded. 
This happens because, unlike projective varieties, the versal deformation space of singularities tends to be infinite dimensional
and many singularities in the versal deformation space will share the same invariants as the central fiber.
In order to fix this issue, we use the following definition.

\begin{definition}\label{def:bounded-deg}
{\em 
Let $\mathcal{C}$ be a class of singularities.
We say that $\mathcal{C}$ 
is \defi{bounded up to degeneration}
if the following condition is satisfied. 
There exists a bounded class $\mathcal{B}$ of singularities 
such that for every element $(X;x)\in \mathcal{C}$ there exists a flat family 
$\mathcal{X}\rightarrow C\ni \{0\}$ 
of singularities for which
$(\mathcal{X}_c;x_c)\simeq (X;x)$ for some $c\in C$ and $(\mathcal{X}_0;x_0)\in \mathcal{B}$.
}
\end{definition}

In other words, we say that a class of singularities is bounded up to degeneration
if the elements of this class
are deformations of singularities
in a bounded class.

\subsection{A smoothness lemma}
We conclude the preliminaries with a smoothness lemma that will be used to construct examples in \Cref{sec:ex}. In particular, certain complete intersections of Fermat hypersurfaces are smooth.  The notation $p_k = p_k(x_0,\ldots,x_N) \coloneqq \sum_{i=0}^N x_i^k$ denotes the $k$-th power sum equation in $N+1$ variables. 

\begin{lemma}[\cite{RY00}]
\label{lem:Fermat_ci_smooth}
For any positive integers $m \leq N-1$, the intersection of Fermat hypersurfaces
$$X \coloneqq \{p_1(x_0,\ldots,x_N) = p_2(x_0, \ldots, x_N) = \cdots = p_m(x_0,\ldots,x_N)\} \subset \pp^N$$
is smooth and irreducible of dimension $N-m$.
\end{lemma}

\begin{proof}
This follows directly from results in \cite{RY00}.  Indeed, the affine cone $C_X$ over the variety $X$ is the subvariety in $\mathbb{A}^{N+1}$ cut out by the same equations.  For $m \leq N-1$, \cite[Lemma 9.4]{RY00} shows that $C_X$ is irreducible of dimension $N-m+1$, hence $X$ is irreducible of the indicated dimension.  Then, \cite[Lemma 9.3]{RY00} shows that $C_X \setminus \{0\}$ is smooth, so $X$ is smooth as well.
\end{proof}

\section{Bounds for symmetric actions}\label{sec:fano-bound}

In this section, we study upper bounds
for symmetric actions on Fano varieties and klt singularities.
First, we show an explicit quadratic upper bound for $k$ 
where $S_k$ is a symmetric group acting
faithfully on an $n$-dimensional Fano variety.

\begin{proposition}\label{prop:ln-bound}
For any integer \(n\geq 1\), let \(p_n\) be the smallest prime greater than \(n+1\). There exists an integer \(m(n)< p_n(n+1)\) such that for any \(n\)-dimensional rationally connected variety \(X\) over a field of characteristic \(0\) and embedding \(S_k\hookrightarrow\Bir X\), we have \(k\leq m(n)\). In particular, for \(n\gg 0\),
\[
m(n)<\left( 1 + \frac{1}{5000\ln^2(n+1)}\right)(n+1)^2.
\] 
\end{proposition}

\begin{proof}
    For a prime \(p\) and integer \(i\geq 1\), let \(W_p(i)\) denote the Sylow \(p\)-subgroups of the symmetric group \(S_{p^i}\). Note that \(W_p(1)=\mathbb Z/p\) and that \(W_p(i)\) is non-abelian for \(i\geq 2\) \cite[Theorem 7.27]{Rotman-groups}.
    
    Let \(p_n\) be the smallest prime greater than \(n+1\). If \(k\geq p_n(n+1)\), then the Sylow \(p_n\)-subgroups of \(S_k\) contain either \((\mathbb Z/p_n)^{\oplus (n+1)}\) or \(W_{p_n}(i)\) for some \(i\geq 2\) as a direct factor \cite[page 176]{Rotman-groups}. Then \(S_k\) contains a \(p_n\)-group that either has rank greater than \(n\) or is non-abelian, so by \cite[Main Theorem]{Xu-p-groups} there does not exist an embedding \(S_k\hookrightarrow\Bir X\). Thus, we have \(k<p_n(n+1)\). For \(n\geq 468991632\) we have \(p_n\leq \left( 1 + \frac{1}{5000\ln^2(n+1)}\right)(n+1)\) by \cite[Corollary 5.5]{Dusart}, so we conclude that \(m(n)<\left( 1 + \frac{1}{5000\ln^2(n+1)}\right)(n+1)^2\).
\end{proof}

Then, the proof of \Cref{introthm:quadratic-bound-Fano} follows.

\begin{proof}[Proof of \Cref{introthm:quadratic-bound-Fano}]
This follows by taking the limit in the statement of~\Cref{prop:ln-bound}.
\end{proof}

Now, we turn to symmetric actions on klt singularities.
We provide an upper bound for $k$,
where $S_k$ is a symmetric group acting faithfully on a $n$-dimensional klt singularity.

\begin{proof}[Proof of \Cref{introthm:quadratic-bound-klt}]
Let $(X;x)$ be an $n$-dimensional klt singularity
and $S_k$ a symmetric group acting on $(X;x)$.
Let $\pi\colon X\rightarrow Y$ be the quotient of $X$ by $S_k$.
Let $B_Y$ be the divisor with standard coefficients for which $\pi^*(K_Y+B_Y)=K_X$.
Then, the pair $(Y,B_Y;y)$ is klt, where $y = \pi(x)$.
Let $\varphi_Y\colon Y' \rightarrow Y$ be a projective birational morphism satisfying the following conditions:
\begin{itemize}
\item $\varphi_Y$ extracts a unique prime divisor $E'$ over $y$, 
\item the pair $(Y',E'+{\varphi_Y}^{-1}_*B_Y)$ has plt singularities, and 
\item the divisor $-(K_{Y'}+E'+{\varphi_Y}^{-1}_*B_Y)$ is ample over $Y$.
\end{itemize}
This projective birational morphism exists by~\cite[Lemma 1]{Xu14}.
Let $\varphi_X \colon X'\rightarrow X$ be the projective birational morphism obtained by base change
and $\pi'\colon X'\rightarrow Y'$ the corresponding quotient map.
Let $F$ be the reduced preimage of $E'$ on $X'$.
Then, the pair $(X',F)$ is plt and
$-(K_{X'}+F)$ is ample over $X$. 
By the connectedness of log canonical centers,
we can conclude that $F$ is prime.
Indeed, by contradiction, let $F=\sum_{i=1}^k F_i$
and assume that $k\geq 2$.
By~\cite[Connectedness Principle]{FS20}, we conclude that $F$ is connected over $X$ so 
there are two components $F_i$ and $F_j$ that intersect.
As an intersection of log canonical centers is a union of log canonical centers (see~\cite[Theorem 1.1.(ii)]{Amb11}), we are led to a contradiction of the fact that $(X',F)$ is plt.
Thus, $F$ is prime.
By construction, the projective birational morphism $X'\rightarrow X$ is $S_k$-equivariant.
Hence, $S_k$ fixes $F$.
By~\Cref{lem:normal+cyclic}, 
we conclude that $S_k$ acts faithfully on $F$.
Note that $F$ is a Fano type variety, 
so it is a rationally connected variety.
We conclude that $S_k$ acts on a rationally connected variety of dimension at most $n-1$.
Hence, the statement follows from \Cref{introthm:quadratic-bound-Fano}
by taking the limit.
\end{proof}

\section{Symmetries of toric varieties}\label{sec:toric}

In this section, we give an upper bound for symmetric actions on complete simplicial toric varieties:
\begin{theorem}\label{thm:toric}
    Let \(X\) be a complete simplicial toric variety of dimension \(n\). Suppose that the symmetric group \(S_k\) acts faithfully on \(X\).  If $n = 1,2,$ or $3$, then $k \leq n+3$; if $n \geq 4$, then $k \leq n+2$.

    These bounds are sharp for each $n$.  If equality is achieved and $n \neq 2,4$, then $X \cong \pp^n$. If $n = 2$, then $k = 5$ if and only if $X \cong \pp^1 \times \pp^1$.  If $n = 4$, then $k = 6$ if and only if $X \cong \pp^4$ or $X \cong \pp^2 \times \pp^2$.
\end{theorem}
The idea of the proof of \Cref{thm:toric} is to use the structure of the automorphism group of a toric variety developed in \cite{Cox95}.  Briefly, symmetries of a toric variety $X$ come from symmetries of the Cox ring of $X$ or symmetries of the fan of $X$. The \(S_{n+2}\)-action on \(\pp^n\) exhibits an example of the first case (\Cref{ex:symm-projective-rep}). An example of the second case is the \(S_n\)-action on \(\prod_{i=1}^n \pp^1\) by permuting the factors. For \(X=\pp^1\times\pp^1\), the \(S_5\)-action is obtained from an \(A_5\)-action on \(\pp^1\) and a \(\zz/2\)-action exchanging the factors.
We first recall the results from \cite{Cox95} that we will need about automorphisms of a toric variety. For a general reference on toric varieties, see \cite{Ful93}.

Throughout this section, let \(X\) be a complete simplicial toric variety of dimension \(n\), defined by a fan \(\Delta\) in \(N=\mathbb Z^n\). Let \(M \coloneqq \Hom_{\mathbb Z}(N,\mathbb Z)\). Let \(T \coloneqq N\otimes_{\mathbb Z}\mathbb C^*\) be the torus acting on \(X\). We'll use \(\Delta(1)\) to denote the set of one-dimensional cones (rays) of \(\Delta\) and \(d = |\Delta(1)|\) for the total number of rays. The free abelian group \(\zz^{\Delta(1)}\) of \(T\)-invariant Weil divisors on \(X\) fits into an exact sequence
\begin{equation}\label{eqn:toric-cl-sequence}\xymatrix{ 1 \ar[r] & M \ar[r] & \zz^{\Delta(1)} \ar[r] & \Cl X \ar[r] & 1}\end{equation} where \(M\to\zz^{\Delta(1)}\) is defined by \(m\mapsto\sum_{\rho\in\Delta(1)}\langle m, n_\rho\rangle D_\rho\). In particular, \(\Cl X\) is a finitely-generated abelian group with \(\rank((\Cl X)_\Q) = d-n\). The \defi{degree} of an element of \(\zz^{\Delta(1)}\) is defined to be its class in \(\Cl X\).

The toric variety $X$ may be constructed as a geometric quotient $(\C^{\Delta(1)} \setminus Z)/G$, where $G$ is the algebraic group defined as $G \coloneqq \Hom_{\Z}(\Cl X,\C^*)$ and $Z$ is the exceptional set defined by the vanishing of a certain monomial ideal. The action of the group $G$ on $\C^{\Delta(1)}$ is induced by the quotient morphism $\Z^{\Delta(1)} \rightarrow \Cl X$.

The coordinate ring \[R \coloneqq \C[x_\rho \mid \rho\in\Delta(1)]\] of the space \(\C^{\Delta(1)}\) acquires a grading from this action by $\Cl X$.  The resulting graded ring, known as the \defi{Cox ring}, plays a major role in the study of toric varieties.  In particular, its structure is closely related to that of the automorphism group of $X$.

Before stating this connection, we'll introduce some more notation related to the Cox ring and that set of rays $\Delta(1)$ of the fan of $X$.  For each $\alpha \in \Cl X$, let $R_{\alpha}$ be the graded piece of $R$ of elements of degree $\alpha$; then \(R=\bigoplus_{\alpha_i} R_{\alpha_i}\). 

We'll pay particular attention to the graded pieces containing variables $x_{\rho}$ for $\rho \in \Delta(1)$.  Indeed, partition \(\Delta(1)\) into disjoint subsets \[\Delta(1)=\Delta_1\sqcup\cdots\sqcup\Delta_s,\] where each \(\Delta_i\) corresponds to a set of variables with the same degree \(\alpha_i\). For each \(\alpha_i\), one may write \(R_{\alpha_i}=R'_{\alpha_i}\oplus R''_{\alpha_i}\) where \(R'_{\alpha_i}\) is spanned by the monomials \(x_\rho\) for \(\rho\in\Delta_i\).

The dimension \(n\) of \(X\) constrains the possible values for the sizes of the $\Delta_i$ in the above partition:
\begin{lemma}\label{lem:ray-partition}
    Write \(d_i=|\Delta_i|\) and \(d=|\Delta(1)|=\sum_{i=1}^s d_i\). The following hold.
    \begin{enumerate}
        \item\label{part:bound-di} $\sum_{i = 1}^s (d_i - 1) \leq n$.  In particular, $d_i \leq n+1$ for each $i$.
        \item \label{part:di-prod-proj-spaces} If $\sum_{i = 1}^s (d_i - 1) = n$, then $X \cong \pp^{d_1-1} \times \cdots \times \pp^{d_s-1}$.
    \end{enumerate}
\end{lemma}
\begin{proof}
    We have that $s \geq \mathrm{rank}((\Cl X)_{\Q}) = d-n$, so $d-s \leq n$, and part~\eqref{part:bound-di} follows from this inequality. For part~\eqref{part:di-prod-proj-spaces}, label the \(d\) rays on \(X\) as \(\rho_1,\ldots,\rho_d\). For each ray \(\rho_j\), let \(D_j\) be the corresponding torus-invariant divisor. For each fixed \(\Delta_i\) with \(d_i\geq 2\), consider the differences \(\{D_j-D_k \mid \rho_j \neq \rho_k \in \Delta_i\}\). Each such difference is in the kernel of the map to \(\Cl X\) in~\eqref{eqn:toric-cl-sequence}, so it is in the image of an element of $M$.  For each $\Delta_i$, there are $d_i - 1$ independent such differences, and the set \(\Sigma\) of differences across all $\Delta_i$ extends to a basis of $\Z^{\Delta(1)}$.  By the assumption, the \(\zz\)-span of \(\Sigma\) has rank \(n=\rank M\); thus, it is equal to the image of \(M\to\zz^{\Delta(1)}\).
    
    We claim that this implies $d_i \geq 2$ for all \(i\). Indeed, any \(m\in M\) has image \(\sum_{j=1}^d \langle m, n_{\rho_j}\rangle D_j\) which must be in the span of \(\Sigma\) by the conclusion of the last paragraph. Hence if \(d_i=1\) for some \(i\) and \(\Delta_i=\{\rho_j\}\), then $D_j$ does not belong to such a difference, so the ray \(\rho_j\) satisfies \(\langle m,n_{\rho_j}\rangle=0\) for all \(m\in M\). This would imply that the ray is $0$, which is impossible.

    For each $\Delta_i$, the sublattice of $\zz^{\Delta(1)}$ generated by $\{\rho \mid \rho \in \Delta_i\}$ intersects the image of $M$ in the $d_1-1$ rank sublattice of the image of $M$ generated by the ray differences.  The preimages $M_i$ of these sublattices in $M$ decompose $M$ as a direct sum, each component of which evaluates to zero identically on any rays not in the corresponding $\Delta_i$.  We also get a dual decomposition $N = \bigoplus_{i=1}^s N_i$.  The fan generated by the $d_i-1$ rays of $\Delta_i$ in $N_i$ is clearly that of $\pp^{d_i-1}$, so $X \cong \pp^{d_1-1} \times \cdots \times \pp^{d_s-1}$.
\end{proof}

To prove \Cref{thm:toric}, we will use the following results of Cox on automorphisms of simplicial toric varieties \cite{Cox95}.  These results realize the automorphism group of $X$ as a quotient of a group of automorphisms of the affine variety $\C^{\Delta(1)} \setminus Z$.  In particular, let $\widetilde{\Aut}^0(X)$ be the centralizer of the group $G$ in the automorphism group of $\C^{\Delta(1)} \setminus Z$, and  $\widetilde{\Aut}(X)$ be the normalizer.

\begin{theorem}[{\cite{Cox95}}]\label{lem:aut-cox-ring}
    Let \(X\) be a complete simplicial toric variety, and let \(R\) be the Cox ring of \(X\).  We denote by $\Aut_g (R)$ the group of automorphisms of this ring preserving the grading. Let \(G=\Hom_{\mathbb Z}(\Cl X,\mathbb C^*)\), and let \(\Aut(N,\Delta)\) be the group of lattice isomorphisms of \(N\) that preserve the fan \(\Delta\).
    \begin{enumerate}
        \item\label{item:G-graded-Auts} There is a natural isomorphism \(\widetilde{\Aut}^0(X) \cong \Aut_g (R)\).  In particular, \(G\) is in the center of \(\Aut_g (R)\).
        \item\label{item:levi-decomposition} \(\Aut_g (R)\) is isomorphic to the semidirect product \(U \rtimes G_s\), where \(U\) is the unipotent radical and \(G_s=\prod_{i=1}^s\GL(R'_{\alpha_i})\). In particular, any finite subgroup of \(\Aut_g (R)\) is conjugate to a subgroup of \(\prod_{i=1}^s \GL(R'_{\alpha_i})\).
        \item\label{item:Aut^0-vs-graded-Auts} The connected component of the identity in \(\Aut(X)\) is \(\Aut^0(X) \cong \Aut_g (R)/G\).
        \item\label{item:toric-aut-component-group} \(\Aut(N,\Delta) \hookrightarrow S_{\Delta(1)}\) and \[ \widetilde{\Aut} (X) / \widetilde{\Aut}^0(X) \cong \Aut(X) / \Aut^0(X)\cong \Aut(N,\Delta) / \prod_{i=1}^s S_{\Delta_i}.\] 
    \end{enumerate}
\end{theorem}

\begin{proof}
    Most of the statements are taken directly from \cite{Cox95}.  Part \eqref{item:G-graded-Auts} is \cite[Theorem 4.2.(iii)]{Cox95}. The assertions of \eqref{item:levi-decomposition} are contained in \cite[Proposition 4.3.(iv)]{Cox95,cox2014erratum}, except the ``in particular" in~\eqref{item:levi-decomposition}, which follows from the structure theory of Lie groups (see, e.g., \cite[Proposition VIII.4.2]{Hochschild81}). Part \eqref{item:Aut^0-vs-graded-Auts} is \cite[Corollary 4.7.(iii)]{Cox95}, and finally part~\eqref{item:toric-aut-component-group} follows from Corollary 4.7.(v) and the proof of Theorem 4.2.(ii) in \cite{Cox95}.
\end{proof}

Now we begin the proof of \Cref{thm:toric}. \Cref{lem:aut-cox-ring}\eqref{item:toric-aut-component-group} shows that an action on \(X\) decomposes into a part in \(\Aut^0(X)\) and an action on the fan. We will consider these two situations separately. First, we consider the case where \(A_k\) or \(S_k\) is a subgroup of \(\Aut^0(X)\).

\begin{lemma}\label{lem:A_k-Aut0}
    Let \(X\) be a complete simplicial toric variety of dimension \(n\), and let \(\Delta_1\sqcup\cdots\sqcup\Delta_s\) be the partition of \(\Delta(1)\) by degrees defined before \Cref{lem:ray-partition}. Let \(k\geq 5\) be an integer, and let \(\Gamma\) be the alternating group \(A_k\) or the symmetric group \(S_k\). If \(\Gamma\leqslant\Aut^0(X)\), then
    \[n \geq \begin{cases} 1 & \text{if }k=5, \\ 2 & \text{if }k=6, \\ 3 & \text{if }k=7, \\ k-2 & \text{if }k \geq 8\end{cases} \text{ if }\Gamma=A_k, \quad \quad n \geq \begin{cases} 3 & \text{if }k=5,6, \\ k-2 & \text{if }k \geq 7\end{cases} \text{ if }\Gamma=S_k.\]
    Furthermore, if equality holds then \(X\cong\pp^n\).
\end{lemma}

\begin{proof}
    
    By \Cref{lem:aut-cox-ring}\eqref{item:G-graded-Auts} and~\eqref{item:Aut^0-vs-graded-Auts}, \(G\) is contained in the center of \(\Aut_g (R)\), and we have an isomorphism \(\Aut^0(X)\cong \Aut_g (R)/G\). This induces a central extension \[\xymatrix{1\ar[r] & K \ar[r] & H \ar[r] & \Gamma \ar[r] & 1}\] with \(K\leqslant G\) and \(H\leqslant\Aut_g (R)\). If \(\Gamma = A_k\), then by \Cref{exmp:central-extn-A_k} we have \[ H \cong \begin{cases} \tilde{A}_k \text{ or } A_k & \text{if } k=5 \text{ or } k\geq 8, \\ \tilde{A}_k, 3\cdot A_k, 2\cdot A_k, \text{ or }A_k & \text{if } k=6,7.\end{cases}\] If \(\Gamma=S_k\), we have \(H\cong S_k\) or \(H\cong\tilde{S}_k\).

    By \Cref{lem:aut-cox-ring}.\eqref{item:levi-decomposition} we may assume \(H\) is contained in \(\prod_{i=1}^s \GL(R'_{\alpha_i})\). Projection onto each factor gives an induced representation \(H\to \GL(R'_{\alpha_i}) = \GL_{d_i}(\C)\). The composition \(H\hookrightarrow\Aut_g (R)\to\Aut^0(X)\) surjects onto \(\Gamma\), so using \Cref{tab:A_k_reps} and \Cref{tab:S_k_reps} we conclude that some \(1\leq i\leq s\) satisfies
    \[d_i \geq \begin{cases} 2 & \text{if }k=5, \\ 3 & \text{if }k=6, \\ 4 & \text{if }k=7, \\ k-1 & \text{if }k\geq 8, \end{cases} \text{ if }\Gamma=A_k, \quad \quad d_i \geq \begin{cases} 4 & \text{if }k=5,6, \\ k-1 & \text{if }k\geq 7, \end{cases} \text{ if }\Gamma=S_k.\]

    By \Cref{lem:ray-partition} we have \(n \geq d_i - 1\), and if equality holds then \(X\cong\pp^n\). This shows the lemma.
\end{proof}

Now we need to deal with the case $S_k \hookrightarrow \Aut(N,\Delta)/\prod_{i = 1}^s S_{\Delta_i}$.

\begin{lemma}\label{lem:S_k-act-fan}
    Let \(n\geq 2\) and \(k\) be integers, and let \(X\) be a complete simplicial toric variety of dimension \(n\). If \(S_k\leqslant\Aut(N,\Delta)/\prod_{i=1}^s S_{\Delta_i}\), then \(k\leq n+1\).
\end{lemma}

\begin{proof}
    For each positive integer $m$, define $I_m \subset \{1,\ldots,s\}$ to be the set of indices $i$ for which $|\Delta_i| = m$. Then \Cref{lem:ray-partition}.\eqref{part:bound-di} implies that $\{1,\ldots,s\} = I_1 \sqcup \cdots \sqcup I_{n+1}$, since all higher $I_m$ must be empty.
    We first claim that the map $S_k \hookrightarrow \Aut(N,\Delta)/\prod_{i = 1}^s S_{\Delta_i}$ induces a natural embedding \[S_k \hookrightarrow S_{I_1} \times \cdots \times S_{I_{n+1}},\] where each $S_{I_m}$ is the symmetric group on partition pieces $\Delta_i$ of $\Delta(1)$ of size $m$.

    Indeed, an element of $\Aut(N,\Delta)$ is an automorphism of the lattice $N$ preserving the fan $\Delta$, so in particular it preserves the linear equivalence of rays in $\Delta$ (see \cite[page 26]{Cox95} for more details).  Hence every member of a collection $\Delta_i$ of linearly equivalent rays is sent to a member of a single collection $\Delta_j$; furthermore, $|\Delta_i| = |\Delta_j|$.  Therefore, mapping $\varphi \in \Aut(N,\Delta)$ to the assignments $i \mapsto j$ defines a group homomorphism $\Aut(N,\Delta) \rightarrow S_{I_1} \times \cdots \times S_{I_{n+1}}$. The kernel of this homomorphism is precisely the subgroup $\prod_{i = 1}^s S_{\Delta_i}$, so it descends to $\Aut(N,\Delta)/\prod_{i = 1}^s S_{\Delta_i} \hookrightarrow S_{I_1} \times \cdots \times S_{I_{n+1}}$.  Composing with the inclusion of the subgroup $S_k$ gives the desired embedding above.
    
    Therefore, $S_k$ acts on each set $I_1, \ldots, I_{n+1}$ of collections of linearly independent rays of a given size.  We assume by way of contradiction that $k \geq n+2$.  Each set $I_m$ for $m \geq 2$ has size at most $n$ by \Cref{lem:ray-partition}.\eqref{part:bound-di}, so \(S_k\) cannot act faithfully on any of these sets.  Therefore, the composite homomorphism
$$S_k \hookrightarrow S_{I_1} \times \cdots \times S_{I_{n+1}} \rightarrow S_{I_1}$$
with the projection onto the first factor must be an injection; that is, $S_k$ acts faithfully on size $1$ linear equivalence classes of rays.  We will use this fact to find a faithful $S_k$ representation of small dimension.

Let \(V\subset N_{\Q}\) be the \(\Q\)-vector space spanned by \(\{\rho\in\Delta_i \mid i\in I_1\}\). There is a restriction homomorphism \[\Aut(N,\Delta)/\prod_{i=1}^s S_{\Delta_i} \to \GL(V).\]
Indeed, for \(\varphi\in\Aut(N,\Delta)\), let \(\varphi_{\Q}\in\GL(N_{\Q})\) denote the extension of \(\varphi\) by scalars to the vector space \(N_{\Q}\). Then \(\varphi_{\Q}(V)=V\) because \(\varphi\) preserves the collection of rays with linear equivalence class of size \(1\). It follows that we have a natural restriction map $\Aut(N,\Delta) \rightarrow \GL(V)$.  Moreover, any element of $\prod_{i=1}^s S_{\Delta_i} \leqslant \Aut(N,\Delta)$ is sent to the identity transformation under this restriction, since it must fix every ray in a spanning set of $V$.

Thus, we have a homomorphism \(S_k\to\GL(V)\). Since we've already shown that the subgroup $S_k \leqslant \Aut(N,\Delta)/\prod_{i=1}^s S_{\Delta_i}$ acts faithfully on the collection of rays $\rho$ with index in $I_1$, the composite homomorphism $S_k \rightarrow \GL(V)$ must be injective.  This proves that $V$ is a faithful representation of $S_k$.  Since \(\dim_{\Q} V \leq \rank N = n\), we must have that $k \leq n+1$ by \Cref{tab:S_k_reps}. This contradicts the assumed bound on $k$.
\end{proof}

Finally, we'll consider the situation where a subgroup $S_k \leqslant \Aut(X)$ has the property $S_k \ \cap \Aut^0(X) = A_k$.

\begin{lemma} \label{lem:switching_indices}
         Let \(n\geq 2\) and \(k \geq 5\) be integers, and let \(X\) be a complete simplicial toric variety of dimension \(n\). Suppose that $S_k \leqslant \Aut(X)$ is a subgroup of automorphisms with the property that $S_k \cap \Aut^0(X) = A_k$, and $S_k $ is not a subgroup of $\Aut^0(X)$.  Then there must exist at least two distinct indices $i$ such that $d_i = |\Delta_i|$ satisfies
    \[d_i \geq \begin{cases} 2 & \text{if }k=5, \\ 3 & \text{if }k=6, \\ 4 & \text{if }k=7, \\ k-1 & \text{if }k\geq 8. \end{cases}\]
    \end{lemma}

    \begin{proof}
    Since we have $S_k \ \cap \ \Aut^0(X) = A_k \leqslant \Aut^0(X) = \Aut_g (R)/G$, as in \Cref{lem:A_k-Aut0}, we have representations of $\tilde{A}_k$ on the factors $\GL(R'_{\alpha_i})$ whose product is the reductive subgroup $G_s$ of \Cref{lem:aut-cox-ring}.\eqref{item:levi-decomposition}.  At least one of these must be faithful.  Therefore, the dimension $d_i = |\Delta_i| = \dim(R'_{\alpha_i})$ must satisfy the inequalities in the lemma for some $i$.

    We will assume that there exactly one index satisfying the inequalities of the lemma, and then derive a contradiction. We may assume this index is \(1\). Then the representation of $\tilde{A}_k$ on each $\GL(R'_{\alpha_i})$ is trivial for $i \geq 2$.  Next, consider the preimage $H$ of the entire $S_k \leqslant \Aut(X)$ inside $\widetilde{\Aut}(X)$, so that $S_k \cong H/G$.  The group $G$ is of multiplicative type, hence reductive, so $H$, being an extension of $S_k$ by $G$, is also reductive.  
    
    We saw in \Cref{lem:aut-cox-ring}.\eqref{item:levi-decomposition} that $\widetilde{\Aut}^0(X)$, the connected component of the identity in $\widetilde{\Aut}(X)$, contains the reductive subgroup $G_s = \prod_{i=1}^s \GL(R_{\alpha_i}')$.  One can find an analogous reductive subgroup of $\widetilde{\Aut}(X)$ as follows.  It was shown in \cite[page 27]{Cox95} that $\widetilde{\Aut}(X)$ is generated by $\widetilde{\Aut}^0(X)$ and elements of the form $P_{\varphi}$, where $\varphi \in \Aut(N,\Delta)$ is an automorphism of the fan.  The automorphism $P_{\varphi}$ is constructed on the level of $\C^{\Delta(1)}$ as the corresponding permutation matrix on rays; this automorphism then descends to the quotient $X = (\C^{\Delta(1)} \setminus Z)/G$ \cite[page 26]{Cox95}. The subgroup generated by $G_s$ and the $P_{\varphi}$ is a  reductive subgroup $G_s'$ of $\widetilde{\Aut}(X)$ with the property that $U G_s' = \widetilde{\Aut}(X)$. (Here $U$ is the unipotent radical of $\widetilde{\Aut}^0(X)$ from \Cref{lem:aut-cox-ring}.\eqref{item:levi-decomposition}; it is also the unipotent radical of $\widetilde{\Aut}(X)$.)  Thus, by \cite[Proposition VIII.4.2]{Hochschild81}, the group $H$ is conjugate to a subgroup of $G_s'$. We may therefore assume \(H\leqslant G_s'\).

    Now pick an transposition $\tau$ of order $2$ in $S_k \leqslant \Aut(X)$, so that $\tau$ and $A_k$ generate the subgroup $S_k$.  For a lift $\tilde{\tau} \in \widetilde{\Aut}(X)$ of $\tau$, we have by assumption that $\tilde{\tau} \in G_s'$.  Since $G_s$ is normal in $G_s'$, we may write $\tilde{\tau}$ as a composition $P_{\varphi} \circ h$, where $\varphi \in \Aut(N,\Delta)$ and $h \in G_s$.  By assumption, $\Delta_1$ is the unique largest piece of the partition, so the permutation on partition pieces that $\varphi$ induces must fix the piece $\Delta_1$.  After changing $\varphi$ by an element $\sigma$ of $\prod_{i=1}^s S_{\Delta_i}$ (the corresponding $P_{\sigma}$ is in $G_s$), we may even assume \(P_\varphi\) induces the identity permutation on $\Delta_1$.  Both $\tilde{\tau}$ and $h$ therefore act by the same linear transformation when restricted to the space $R'_{\alpha_1}$; we shall denote by $\tilde{\tau}'$ the automorphism in $G_s \leqslant \widetilde{\Aut}^0(X)$ that acts by this linear transformation in $R'_{\alpha_1}$ and is constant on all other $x_{\rho}$, $\rho \in \Delta(1)$.

    Let $\tau' \in \Aut(X)$ be the image of $\tilde{\tau}'$.  The point is now to show that $\tau'$ and $A_k$ generate a copy of $S_k$ just as $\tau$ and $A_k$ do, but this time inside of $\Aut^0(X)$.  Indeed, we have that $\tilde{\tau}^{-1} \tilde{\tau}'$ is trivial on $R'_{\alpha_1}$, so its image $\tau^{-1}\tau'$ in $\Aut(X)$ commutes with any $g \in A_k$.  This implies $(\tau')^{-1} g \tau' = \tau^{-1} g \tau \in A_k$ for any such $g$.  Therefore, the group $\Gamma$ generated by $\tau'$ and $A_k$ has order $2 \cdot |A_k| = k!$ and the action by $\tau'$ on the normal subgroup $A_k$ by conjugation is the same as that of $\tau$.  This shows $\Gamma$ has the same semidirect product structure as $S_k$ does, so $\Gamma \cong S_k$.  This contradicts the assumption that there is no embedding $S_k \hookrightarrow \Aut^0(X)$, completing the proof.
    \end{proof}

Putting the above results together, we can now prove \Cref{thm:toric}.
\begin{proof}[Proof of \Cref{thm:toric}]
For each dimension $n$, we may assume that $k$ is at least the upper bound given in the statement of \Cref{thm:toric} (if not, the conclusion holds automatically).  Under this assumption, we show that $k$ must in fact equal this bound and characterize the optimal examples.
    First, we deal with $n = 1$. A one-dimensional normal complete toric variety is isomorphic to $\pp^1$, so $X = \pp^1$ and $S_4 \leqslant \PGL_2(\C)$ is the largest symmetric action.

    From now on, we consider $n \geq 2$ so that we may assume $k \geq 5$.  Therefore, $A_k$ is simple. The cokernel of \(S_k\cap\Aut^0(X) \to S_k\leqslant\Aut(X)\) is either trivial, \(\zz/2\), or \(S_k\). 
    
    If the cokernel is trivial, then we have an embedding $S_k \hookrightarrow \Aut^0(X)$.  Using \Cref{lem:A_k-Aut0}, we may get a bound on $n$.  For $n = 2$, the lemma implies $k < 5$, contradicting the maximality assumption $k \geq 5$.  Therefore, no maximal symmetric actions on toric surfaces occur in the case of trivial cokernel.  For $n = 3$, any embedding $S_k \hookrightarrow \Aut^0(X)$ satisfies $k \leq 6$, and for $n \geq 4$ we must have $k \leq n+2$.  So, for all \(n\geq 2\), our assumption that $k$ is maximal means that the inequalities are equalities and $X \cong \pp^n$, again by \Cref{lem:A_k-Aut0}.  In particular, $\pp^n$ achieves the optimal bound in dimensions $n \geq 3$.

    If the cokernel of \(S_k\cap\Aut^0(X) \to S_k\leqslant\Aut(X)\) is all of $S_k$, then we have \(S_k\hookrightarrow \Aut(X)/\Aut^0(X) \cong \Aut(N,\Delta)/\prod_{i=1}^s S_{\Delta_i}\).  We claim that this case produces no maximally symmetric examples.  Indeed, \Cref{lem:S_k-act-fan} shows that for each $n$, $k \leq n+1$.  Either way, $k$ falls short of the maximum possible value laid out in \Cref{thm:toric}.
    
    Finally, we consider the case where the cokernel of \(S_k\cap\Aut^0(X) \to S_k\leqslant\Aut(X)\) is \(\zz/2\).  We can suppose without loss of generality that there is no embedding $S_k \hookrightarrow \Aut^0(X)$, or else we'd be back in the trivial cokernel case.  This situation is characterized by \Cref{lem:switching_indices}.
    
    Begin with the $n = 2$ case.  If $k \geq 6$, we'd have by \Cref{lem:switching_indices} that $d_1,d_2 \geq 3$, contradicting \Cref{lem:ray-partition}.\eqref{part:bound-di}.  This leaves only $k = 5$ to consider.  \Cref{lem:switching_indices} gives $d_1, d_2 \geq 2$, so in fact $d_1 = d_2 = 2$ or else we'd again contradict \Cref{lem:ray-partition}.\eqref{part:bound-di}.  Thus, $X \cong \pp^1 \times \pp^1$ by \Cref{lem:ray-partition}.\eqref{part:di-prod-proj-spaces}. On the other hand, we know that $S_5$ acts faithfully on the toric variety $\pp^1 \times \pp^1$ (see \Cref{ex:optimal_example} for $n = 2$).  Therefore, $\pp^1 \times \pp^1$ is the unique optimal example for $n = 2$.

    Now consider $n = 3$. If $k \geq 6$, then \Cref{lem:switching_indices} shows that we'd have (without loss of generality) $d_1, d_2 \geq 3$, contradicting \Cref{lem:ray-partition}.\eqref{part:bound-di}.  Therefore, we get no new maximal examples.

    For $n = 4$, $k \geq 7$, we'd have $d_1,d_2 \geq 4$, once again a contradiction.  The remaining possibility is $k = 6$, where we need $d_1 = d_2 = 3$.  This implies $X \cong \pp^2 \times \pp^2$ by \Cref{lem:ray-partition}\eqref{part:di-prod-proj-spaces}.
    Conversely, we claim that $\Aut(\pp^2 \times \pp^2) \cong \PGL_3 \wr \ \zz/2$ \cite[Theorem 1]{LL21} contains a copy of $S_6$. This is because $A_6 \leq \PGL_3(\C)$ and $S_6$ is a semidirect product of $A_6$ and $\zz/2$.  This semidirect product is a subgroup in the wreath product generated by a twisted diagonal embedding of $A_6$ and the transposition of factors. Therefore, $\pp^2 \times \pp^2$ is another optimal example for $n = 4$.

    For $n = 5$, the assumption $k \geq 7$ means $d_1,d_2 \geq 4$, a contradiction.  Finally, when $n \geq 6$, we can assume $k \geq n+2$ so we'd have $d_1,d_2 \geq n+1$ so $\sum_{i=1}^s (d_i-1) \geq 2n > n$.  In summary, no maximal examples can occur in this case for $n \geq 5$.
\end{proof}

\section{Symmetries of weighted complete intersections}
\label{sect:wci}

In this section, we find the largest symmetric group which can act on a Fano or Calabi--Yau variety which is a quasismooth weighted complete intersection of dimension $n$. We'll first review a few key definitions.

We say that a weighted projective space $\pp \coloneqq \pp(a_0,\ldots,a_N)$ is \defi{well-formed} if $\gcd(a_0,\ldots, \hat{a}_i,\ldots,a_N) = 1$ for all $1 \leq i \leq N$.  A subvariety $X$ of $\pp$ is \defi{well-formed} if $\pp$ is well-formed and
$$\dim X - \dim(X \cap \mathrm{Sing}(\pp)) \geq 2,$$
where by convention the empty set has dimension $-1$.  The subvariety $X$ is \defi{quasismooth} if its preimage in $\mathbb{A}^{N+1} \setminus \{0\}$ is smooth.  We'll always work with quasismooth weighted complete intersections throughout this paper.
For a thorough introduction to weighted complete intersections, see \cite{Iano-Fletcher}.

The main theorem of this section is as follows.

\begin{theorem}
\label{thm:max_symmetric_action}
Let $X$ be a quasismooth weighted complete intersection of dimension $n$.  Suppose that the symmetric group $S_k$ acts faithfully on $X$.  The following hold:
\begin{enumerate}
    \item\label{part:max_symmetric_action-Fano} If $X$ is Fano, then
    $$k \leq n + \left\lceil \frac{1 + \sqrt{8n + 9}}{2} \right\rceil.$$
    This bound is sharp for every $n$.
    \item If $X$ is Calabi--Yau, then
    $$k \leq n + \left\lfloor \frac{1 + \sqrt{8n+9}}{2}\right\rfloor + 1.$$
\end{enumerate}
\end{theorem}
\Cref{ex:optimal_example} shows that \eqref{part:max_symmetric_action-Fano} is sharp in every dimension. In the Calabi--Yau case, the bound $k \leq 4$ given by (2) for $n = 1$ is not sharp because $S_3$ is the largest symmetric action on a smooth elliptic curve, by the proof of \Cref{prop:low_dim_wci} below.  It is unclear whether (2) is always sharp in higher dimensions (see \Cref{rem:CY_sharp}).

For succinctness, we'll use the following abbreviations for the functions above throughout the section.
\begin{align*}
    c_{\mathrm{Fano}}(n) & \coloneqq  n + \left\lceil \frac{1 + \sqrt{8n + 9}}{2} \right\rceil, &
    c_{\mathrm{CY}}(n) & \coloneqq n + \left\lfloor \frac{1 + \sqrt{8n+9}}{2}{}\right\rfloor + 1.
\end{align*}

\begin{remark}
    {Notice that these two functions satisfy $c_{\mathrm{CY}}(n) \geq c_{\mathrm{Fano}}(n)$, they  never differ by more than $1$, and they are equal unless the fractional expression is an integer.  It is also true that $c_{\mathrm{Fano}}(n-1)$ and $c_{\mathrm{CY}}(n-1)$ are both strictly smaller than $c_{\mathrm{Fano}}(n)$ for all $n$.  Since we expect \Cref{ex:optimal_example} to be a maximally symmetric Fano for each $n$, this suggests that the hypothesis of \Cref{thm:birat_boundedness} is likely to hold.  It also provides some evidence that the proof of \Cref{introthm:S8-4-fold-bounded} in~\Cref{sec:sym-bound} should extend to higher dimensions, because we expect that the maximal $S_k$ which can act on a Fano variety of dimension $n$ cannot act faithfully on either a Calabi--Yau or a Fano variety of dimension $n-1$.  This in turn is one of the key inductive steps to generalizing \Cref{introthm:S8-4-fold-bounded} (see the remarks before \Cref{quest:max-symm-bound}).}
\end{remark}

Throughout Sections~\ref{sect:wci} and~\ref{sect:max_sym_vars}, we'll use the following notation:

\begin{notation}
Let $X\coloneqq X_{d_1,\ldots,d_m} \subset \pp \coloneqq \pp(a_0,\ldots,a_N)$ be a quasismooth weighted complete intersection defined by $m$ weighted homogeneous equations $f_1, \ldots, f_m$, of degrees $d_1, \ldots, d_m$, respectively.  The dimension of $X$ is $n \coloneqq N - m$. Assume the symmetric group \(S_k\) acts faithfully on \(X\).
\end{notation}

We'll first deal with some low-dimensional cases that are known via other means, so that we may exclude them later.

\begin{proposition}
\label{prop:low_dim_wci}
Let $X$ be a quasismooth weighted complete intersection which is a) Fano of dimension $n \leq 3$, or b) Calabi--Yau of dimension $n \leq 2$.  Suppose $X$ has a faithful action of $S_k$.  Then the upper bounds in \Cref{thm:max_symmetric_action} hold.
\end{proposition}

\begin{proof}
 We may always replace $X$ with a well-formed quasismooth complete intersection which is isomorphic \cite[Lemma 2.3]{PrSFano}, so assume $X$ is well-formed. We'll consider the statements for Fano and Calabi--Yau varieties separately. Begin with case~\eqref{part:max_symmetric_action-Fano}, where $X$ is Fano.

If the dimension of $X$ is $1$, then $X \cong \pp^1$ since it is a klt Fano variety, and it's well-known that $S_4$ is the largest symmetric group that embeds in $\PGL_2(\C)$ (see, e.g., \cite{beauville-PGL2}).  Since $c_{\mathrm{Fano}}(1) = 4$, this proves the theorem in this case.

When $n = 2$, $X$ is rational, so $S_k$ embeds in the Cremona group $\Cr(2)$.  The finite subgroups of $\Cr(2)$ have been classified (see \cite{DI09}); the largest symmetric group action that appears is by $S_5$, which again agrees with $c_{\mathrm{Fano}}(2) = 5$.

Finally, when $n = 3$, a resolution of singularities of $X$ is a rationally connected variety, so $S_k \leqslant \mathrm{Bir}(V)$ for $V$ some rationally connected threefold.  A result of Prokhorov shows that for $k \geq 8$, $S_k$ does not admit an embedding into $\mathrm{Bir}(V)$ for $V$ any rationally connected threefold \cite[Proposition 1.1]{Prokhorov_space_Cremona}. Since $c_{\mathrm{Fano}}(2) = 7$, this proves the bound for $n = 3$.

We next turn to the Calabi--Yau case.  If $\dim X = 1$, then $X$ is a smooth genus $1$ curve. Its automorphism group is a semidirect product of the automorphism group of an elliptic curve with the (abelian) group on translations of $X$.  Since the automorphism group of a complex elliptic curve is cyclic of order $2, 4,$ or $6$, the largest possible symmetric group action on $X$ is by $S_3$.  This $S_3$ is in fact achieved, for instance by the permutation of variables on the cubic curve $X = \{x^3 + y^3 + z^3 = 0 \} \subset \pp^2$. We have $3 < c_{\mathrm{CY}}(1) = 4$.

Let $n = 2$.  We have that $c_{\mathrm{CY}}(2) = 6$.  By the adjunction formula, $K_X \cong \mathcal{O}_X$.  Since $K_X$ is Cartier, $X$ has canonical singularities.  Suppose that $\tilde{X} \rightarrow X$ is a minimal resolution of singularities; then $\tilde{X}$ is an abelian surface or a smooth K3 surface and the action of $S_k$ lifts to $\tilde{X}$ \cite[Proposition 2.2]{ETWindex}.  If $\tilde{X}$ is an abelian surface, then $\Aut(\tilde{X})$ is a semidirect product of the (abelian) group of translations with the subgroup $\Aut(\tilde{X},0)$ which preserves the identity point.  When $k > c_{\mathrm{CY}}(2) = 6$, $A_k \leqslant S_k$ is simple, so it must embed in $\Aut(\tilde{X},0)$.  This is impossible by the classification of automorphism groups of complex tori of dimension $2$ \cite{Fuj88}.

If instead $\tilde{X}$ is a K3 surface, then the action of any finite subgroup $H \leqslant \Aut(\tilde{X})$ on the one-dimensional vector space $H^0(\tilde{X},K_{\tilde{X}}) \cong \C$ gives an exact sequence 
\[\xymatrix{
1\ar[r] &  H_{\rm symp}
\ar[r] &  H \ar[r] &  
\zz/m \ar[r] &  1, }
\]
where \(m\) is a positive integer and $H_{\rm symp}$ is the kernel of the representation, which acts by symplectic automorphisms \cite{Nikulin79}.  As above, if $S_k \leqslant \Aut(\tilde{X})$ for $k > 6$, we'd have a symplectic group of automorphisms on a K3 surface isomorphic to $A_k$.  This is impossible by the classification of finite symplectic actions on such surfaces \cite{Muk88}, proving the required inequality on $k$.
\end{proof}

Before proving \Cref{thm:max_symmetric_action} in higher dimensions, we show some lemmata that we will need in the proof.
The first reduction step is to show that we may assume that no defining equation $f_i$ of \(X\) contains a linear term, i.e., $x_j$ is not a monomial in $f_i$ for any $j,i$.
\begin{lemma}\label{lem:no-linear-term}
    Let \(X_{d_1,\ldots,d_m} \subset \mathbb P(a_0,\ldots,a_N)\) be a quasismooth weighted complete intersection of dimension at least $3$. Then there is a quasismooth, well-formed weighted complete intersection \(X'\subset\mathbb P(a'_0,\ldots,a'_{N'})\) that is isomorphic to \(X\) and such that none of the equations defining \(X'\) contains a linear term.
\end{lemma}

\begin{proof}
    The argument is the same as in \cite[Proposition 2.9]{PrSFano}. Indeed, suppose without loss of generality that $x_N$ is a monomial in $f_m$, so that $f_m = x_N - g$ for some polynomial $g$ depending only on the other variables $x_0, \ldots, x_{N-1}$.  After a change of variables $x_N - g \mapsto x_N$, we may assume that $f_m = x_N$, and that no other $f_i$ contain the variable $x_N$.  Indeed, if another $f_i$ does contain $x_N$, we may modify $f_i$ by subtracting a multiple of $f_m = x_N$ to eliminate that term.  This changes the defining equations, but not the degrees, nor the ideal defining the complete intersection.  Since the equation $x_N = 0$ cuts out $\pp(a_0,\ldots,a_{N-1}) \subset \pp$, $X$ is a codimension $m-1$ weighted complete intersection $X'$ in this smaller weighted projective space.  It follows from the quasismoothness of $X$ that $X'$ is also quasismooth.  Though it may happen that $\pp(a_0,\ldots,a_{N-1})$ is not well-formed, $X'$ will be isomorphic to yet another quasismooth complete intersection in a well-formed weighted projective space $\mathbb P(a'_0,\ldots,a'_{N-1})$ by \cite[Lemma 2.3]{PrSFano}.  If the resulting equations contain a linear term, we may repeat this process until the assumptions are satisfied.
\end{proof}

Using this lemma, we can assume that every $X$ we consider in the proof of \Cref{thm:max_symmetric_action} will be quasismooth and well-formed.  This in particular means that the adjunction formula holds for $X$, i.e. $K_X \cong \mathcal{O}_X(d_1 + \cdots + d_m - a_0 - \cdots - a_N)$ \cite[Theorem 3.3.4]{Dolgachev}.

A quasismooth weighted complete intersection with no linear terms must also satisfy certain conditions on degrees \cite[Lemma 18.14.(i)]{Iano-Fletcher}:

\begin{lemma}
\label{IF_lemma}
Let $X_{d_1,\ldots,d_m} \subset \pp(a_0,\ldots,a_N)$ be a well-formed quasismooth complete intersection such that none of the equations defining $X$ contains a linear term.  Rearrange degrees and weights such that $d_1 \leq \cdots \leq d_m$ and $a_0 \leq \cdots \leq a_N$.  Then the following inequalities hold.
\begin{enumerate}
    \item\label{part:IF_lemma-1} \(d_{m-j} > a_{N-j}\) for all \(0\leq j\leq m-1\).
    \item\label{part:IF_lemma-2} If \(m\geq \dim X + 1\), then \(d_{m-j}-a_{N-j}\geq a_{\dim X-j}\) for all \(0\leq j\leq\dim X\).
\end{enumerate}
\end{lemma}

\begin{proof}
Part~\eqref{part:IF_lemma-1} is \cite[Lemma 18.14.(i)]{Iano-Fletcher}.
The statement of \cite[Lemma 18.14.(i)]{Iano-Fletcher} assumes that the complete intersection $X$ is not the intersection of a linear cone with other hypersurfaces, i.e., that $d_i \neq a_j$ for any $i$ and $j$.  However, their proof only requires that no linear term appears in any of the equations $f_1, \ldots, f_m$ defining $X$, which is precisely what we assumed.

Part~\eqref{part:IF_lemma-2} is \cite[Proposition 3.1.(2)]{CCC}. (Once again, the statement of \cite[Proposition 3.1.(2)]{CCC} assumes that the complete intersection $X$ is not the intersection of a linear cone with other hypersurfaces, i.e., that $d_i \neq a_j$ for all $i$ and $j$, but the same comment made in part~\eqref{part:IF_lemma-1} shows that the proof extends to our situation.)
\end{proof}

Next, we bound the codimension of Fano and Calabi--Yau weighted complete intersections satisfying the conditions above on linear terms.

\begin{lemma}
\label{lem:codimension_bound}
Suppose that $X_{d_1,\ldots,d_m} \subset \pp(a_0,\ldots,a_N)$ is a well-formed quasismooth complete intersection such that none of the equations defining $X$ contains a linear term. 
\begin{enumerate}
    \item If $X$ is Fano, then the codimension $m$ satisfies $m < (N+1)/2$.  Equivalently, the dimension $n$ of $X$ satisfies $2n+2 > N+1$.
    \item If $X$ is Calabi--Yau, then the codimension $m$ satisfies $m \leq (N+1)/2$. Equivalently, the dimension $n$ of $X$ satisfies $2n+2 \geq N+1$.
\end{enumerate}
\end{lemma}

\begin{proof}
After reordering, we may assume $d_1 \leq \cdots \leq d_m$ and $a_0 \leq \cdots \leq a_N$. Suppose that $m \geq (N+1)/2$ so that $m \geq N-m + 1 = \dim X + 1$.  If this does not hold, then the conclusion of either part of the lemma is already true, and there is nothing to prove.  In the case $m \geq (N+1)/2$, both parts of \Cref{IF_lemma} apply.

Since $d_{m-j} \geq a_{N-j} + a_{\dim X-j}$ for $j = 0,\ldots,\dim X$ by \Cref{IF_lemma}.\eqref{part:IF_lemma-2}, we have that 
\begin{equation}
\label{eq:deg_sum_1}
    d_m + d_{m-1} + \cdots + d_{m-\dim X} \geq a_N + \cdots + a_{N - \dim X} + a_{\dim X} + \cdots + a_0.
\end{equation}
In the case that we have a strict inequality $m > \dim X + 1$, there are more degrees $d_i$ that have not appeared in the inequality above. We may now apply \Cref{IF_lemma}.\eqref{part:IF_lemma-1} to the remaining degrees (i.e., for indices $j = \dim X + 1, \ldots, m-1$).  Summing these inequalities gives that
\begin{equation}
\label{eq:deg_sum_2}
    d_{m-\dim X-1} + \cdots + d_1 > a_{N-\dim X-1} + \cdots + a_{N-m+1}.
\end{equation}
Because $N-m+1 = \dim X + 1$, all the weights of $\pp$ appear on the right-hand side of either \eqref{eq:deg_sum_1} or \eqref{eq:deg_sum_2}.  Adding these two inequalities thus yields $d_1 + \cdots + d_m \geq a_0 + \cdots + a_N$. Furthermore, this inequality is strict unless $m = \dim X + 1$ exactly.  

By the adjunction formula \cite[Theorem 3.3.4]{Dolgachev}, $K_X = \mathcal{O}_X(d_1 + \cdots + d_m - a_0 - \cdots - a_N)$.  Under the assumption $m \geq (N+1)/2$, we've therefore shown that $d_1 + \cdots + d_m - a_0 - \cdots - a_N$ is nonnegative and hence $X$ is not Fano.  This completes the proof of part \eqref{part:IF_lemma-1} of the lemma.

The complete intersection $X$ is Calabi--Yau if and only if $d_1 + \cdots + d_m = a_0 + \cdots + a_N$.  Again under the assumption $m \geq (N+1)/2$, we've shown that this can only occur if we have equality $m = (N+1)/2$.  This proves \eqref{part:IF_lemma-2} of the lemma.
\end{proof}

Having finished the preliminaries, we now begin the main part of the proof of \Cref{thm:max_symmetric_action}. We'll next prove some general properties of higher-dimensional weighted complete intersections with large symmetric actions, which will be key to finishing the proof of~\Cref{thm:max_symmetric_action}.  Indeed, the following lemma assumes that the $S_k$-action on $X$ of dimension $n$ satisfies $k \geq c_{\mathrm{Fano}}(n)$ in the Fano case or $k > c_{\mathrm{CY}}(n)$ in the Calabi--Yau case.  These assumptions put big constraints on $X$, and we will show later that strict inequality will lead to a contradiction. For Fano weighted complete intersections, we include the case of equality $k = c_{\mathrm{Fano}}(n)$ because it will be useful for the classification of maximal examples in \Cref{sect:max_sym_vars}.

\begin{lemma}
\label{lem:large_dim_wci_facts}
Let $S_k$ act faithfully on a well-formed quasismooth weighted complete intersection $X$ of dimension $n$ such that no equation of $X$ contains a linear term.  Suppose that either a) $X$ is Fano, $n \geq 4$, and $k \geq c_{\mathrm{Fano}}(n)$; or b) $X$ is Calabi--Yau, $n \geq 3$, and $k > c_{\mathrm{CY}}(n)$.  Then, after an appropriate change of variables, the following properties hold:
    \begin{enumerate}
    \item\label{part:wci_Ak-standard} The subgroup $A_k \leqslant S_k$ acts by the standard representation in the first $k-1$ variables $x_0, \ldots, x_{k-2}$, which all have the same weight $b$, and acts trivially on the remaining variables $x_{k-1},\ldots,x_N$.
    \item\label{part:wci-Sk-ideal} The equations $f_1,\ldots,f_m$ are contained in the ideal $\langle \sigma_2,\ldots, \sigma_k, V, x_{k-1}, \ldots, x_N \rangle$, where $\sigma_2, \ldots, \sigma_k$, are the elementary symmetric polynomials in $x_0, \ldots, x_{k-2}, y \coloneqq -x_0 - \cdots - x_{k-2}$, and $V$ is the Vandermonde polynomial in these variables.
    \item\label{part:wci-Sk-total-deg} Any collection of $\alpha$ equations among $\{f_1,\ldots,f_m\}$, for any $1 \leq \alpha \leq m$, have total degree at least $\left(\frac{(\alpha + 1)(\alpha+2)}{2}-1 \right)b$.
\end{enumerate}
\end{lemma}

\begin{proof}
Since $n \geq 3$, the subgroup $S_k \leqslant \Aut(X)$ lifts to a subgroup $S_k \leqslant \Aut(\pp)$ by \cite[Theorem 1.3]{PrS2} (here we use that $S_k$ is a reductive group in characteristic zero).

By \Cref{lem:aut-cox-ring} the automorphism group $\Aut(\pp)$ of weighted projective space is described by an exact sequence
\[\xymatrix{1 \ar[r] &  \C^* \ar[r] &  \Aut(S) \ar[r] &  \Aut(\pp) \ar[r] &  1,}\]
where $S = \C[x_0, \ldots, x_{N}]$ is the polynomial ring with each variable $x_i$ of weight $a_i$ and $\Aut(S)$ is the group of graded automorphisms of this ring.
The subgroup $\C^*$ is the group of ``scalar transformations" which for each $i$ map $x_i \mapsto t^{a_i}x_i$, for some $t \in \C^*$.  Since the Schur multiplier $H^2(S_k,\C^*)$ is $\mathbb{Z}/2$ for $k \geq 4$, the map $S_k \rightarrow \Aut(\pp)$ lifts to $\tilde{S}_k \rightarrow \Aut(S)$, where $\tilde{S}_k$ is one of the two representation groups of $S_k$, a central extension of order $2$.

\Cref{lem:aut-cox-ring}.\eqref{item:levi-decomposition} (see also the proof of \cite[Lemma 3.5]{Esser}) shows that any finite subgroup of $\Aut(S)$ is conjugate to one inside the reductive subgroup $\prod_{\ell} \GL_{N_{\ell}}(\C) \leqslant \Aut(S)$ given by the group of automorphisms that do not ``mix" variables with weights of different sizes.  Here $\sum_{\ell} N_{\ell} = N + 1$ is the total number of weights.  (For example, when $\pp = \pp(5,5,2,2,2)$, we have a $\GL_2(\C)$ acting on the first two variables, and a $\GL_3(\C)$ on the last three.)  Projection to each factor $\GL_{N_{\ell}}(\C)$ gives a linear representation of $\tilde{S}_k$.  Since the original map $S_k \rightarrow \Aut(\pp)$ was injective, at least one of these representations must be a faithful linear representation of $S_k$ or $\tilde{S}_k$.

Let $I \subset \C[x_0,\ldots,x_N]$ be the weighted homogeneous prime ideal defining the weighted complete intersection $X$.  Then $I$ is invariant under the $\tilde{S}_k$-action.  By Nakayama's lemma, $I/\mathfrak{m}I$ is a $\C$-vector space with dimension the minimal number of generators of $I$, where $\mathfrak{m}$ is the irrelevant ideal of the graded polynomial ring $\C[x_0,\ldots,x_N]$.  But the minimal number of generators of $I$ is $m$, the codimension of $X$.  This is at most $n+1$ by \Cref{lem:codimension_bound} and hence less than $k-1$, so the action of $\tilde{S}_k$ on $I/\mathfrak{m}I$ is trivial up to sign by the classification of $\tilde{S}_k$ representations (see \Cref{tab:S_k_reps}).  Thus, we may choose a set of weighted homogeneous generators $f_1,\ldots,f_m$ for $I$ such that each $f_i$ is $\tilde{S}_k$-invariant up to sign. 

We saw above that, after an appropriate change of variables in $\pp$, the $S_k$-action on $X$ lifts to an $\tilde{S}_k$-action on $\C[x_0,\ldots,x_N]$ which acts linearly on each vector space \(\C^{N_{\ell}}\) of variables of each given weight. By \Cref{lem:codimension_bound}, $X$ is a complete intersection in $\pp(a_0,\ldots,a_N)$ where $N+1 < 2n+2$ (in the Fano case) or $N+1 \leq 2n+2$ (in the Calabi--Yau case). 

\textit{Claim:} In this setting, we have an irreducible $(k-1)$-dimensional linear representation of $\tilde{S}_k$ inside a space \(\C^{N_{\ell}}\) of variables of the same weight in $\C[x_0,\ldots,x_N]$.

To show the claim, we consider several different cases.
The total number $N+1$ of weights is at most $2n+2$. Thus, \Cref{sym_group_reps} implies that if $n \geq 4$ and we're not in the special case $n = 4$ and $k = 8$, the only irreducible representations which are of small enough dimension to comprise the $\tilde{S}_k$-action on some \(\C^{N_{\ell}}\) are of dimension $1$ and $k-1$.  They cannot all be dimension $1$ or else $S_k$ would not act faithfully on $X$.  We conclude that there is a $(k-1)$-dimensional representation on the weights.

It remains to consider the exceptional cases (1) \(n=3\) and $X$ Calabi--Yau,  and (2) \((n,k)=(4,8)\).  First consider when $X$ is Calabi--Yau and $n = 3$.  The assumptions of the lemma mean $k \geq 8 =c_{\mathrm{CY}}(3) + 1$.  If $k > 8$, all representations of $\tilde{S}_k$ other than those of dimension $1$ and $k-1$ have dimension larger than $2n+2 \geq N+1$ as above.  When $k = 8$, we could conceivably have that $N+1 = 8$ and that $\tilde{S}_8$ acts faithfully by the basic spin representation of dimension $8$.  This would mean that all $8$ weights $a_0,\ldots,a_7$ are equal, so actually $X \subset \pp^7$ is a smooth complete intersection of codimension $4$.  The only way this is possible (since there are no linear equations) is if $X$ is a $(2,2,2,2)$-complete intersection.  But up to scaling there is only one polynomial of degree $2$ which is $\tilde{S}_8$-invariant up to sign (see \Cref{sym_group_reps}), a contradiction.

When \(n=4\) and \(X\) is Calabi--Yau, then the assumption implies \(k>8\). So the remaining exceptional case is when $X$ is Fano and $n = 4$. \Cref{lem:codimension_bound} guarantees that the number of weights is less than $2n+2 = 10$.  If there are $8$ weights and we want to fit a basic spin representation of $\tilde{S}_8$, we have that $X \subset \pp^7$ again, and the invariant polynomials do not have low enough degree as above.  If there are $9$ weights, and the first $8$ are part of the faithful spin representation of $\tilde{S}_8$, then $X \subset \pp(1^{(8)},a)$ has codimension $4$.  At least one equation includes the variable corresponding to $a$ and all four must involve invariant (up to sign) polynomials in the first $8$ variables.  The total degree is therefore more than $a + 2 + 8 + 8$, so $X$ could not be Fano. This shows the claim.

Therefore, we have a $(k-1)$-dimensional linear representation inside a space \(\C^{N_{\ell}}\) of variables of the same weight in $\C[x_0,\ldots,x_N]$. Since $2(k-1) = 2k-2 > 2n+2 \geq N+1$, there is exactly one \(\ell\) with the above property.  After reordering the variables, we can conclude the following: $\tilde{S}_k$ acts by the standard representation of $S_k$ (or its tensor product with the sign representation) on the variables $x_0,\ldots,x_{k-2}$, which all must be of the same weight $b$.  In addition, it acts trivially, or by the sign representation, on all other variables $x_{k-1},\ldots,x_N$, which could all be of different weights. Since all the representations that appear are actually $S_k$ representations rather than just $\tilde{S}_k$ representations, we'll only work with $S_k$ from now on.  In particular, we now know that $A_k \leqslant S_k$ acts by the standard representation in the first $k-1$ variables and acts trivially on the remaining ones. This completes the proof of~\eqref{part:wci_Ak-standard}.

We saw above that each of $f_1,\ldots,f_m$ must be $S_k$-invariant up to sign. In particular, all these equations are $A_k$-invariant, so $f_1,\ldots,f_m \in \langle \sigma_2,\ldots,\sigma_k,V,x_{k-1},\ldots,x_N \rangle$, where $\sigma_2,\ldots,\sigma_k$ are the elementary symmetric polynomials in $x_0,\ldots,x_{k-2},y \coloneqq -x_0-\cdots-x_{k-2}$, and $V$ is the Vandermonde polynomial in these variables. Indeed, for the usual permutation representation of $A_k$ on $\C^k$, the invariant ring would be generated by the first $k$ elementary symmetric polynomials and the Vandermonde polynomial.  The standard representation is the subspace of the permutation representation $\C^k$ where the variables add to zero, so the invariants are given as above, with $\sigma_1$ omitted, and $x_0 + x_1 + \cdots + x_{k-2} + y = 0$.  This shows~\eqref{part:wci-Sk-ideal}.

In order to prove~\eqref{part:wci-Sk-total-deg}, first observe as above that the assumption \(k\geq c_{\mathrm{Fano}}(n)\) or $k > c_{\mathrm{CY}}(n)$ implies $k-1 > \frac{N+1}{2}$, so more than half the total variables belong to the set permuted by $S_k$.  Also, $k > n$, so the codimension satisfies $m = N-n \geq N-k$; that is, there are more equations than there are variables not belonging to the permutation action. From here, we make a few additional simple observations.

Each of the polynomials $f_1,\ldots,f_m$ must involve some variable $x_0,\ldots,x_{k-2}$ or else $X$ would fail to be quasismooth.  This follows from the same type of arguments as the proof of \Cref{lem:codimension_bound}.  Indeed, suppose one of the equations, say $f_1$, does not include any of $x_0,\ldots,x_{k-2}$.  Then let $\Pi$ be the $(k-1)$-plane in $\mathbb{A}^{N+1}$ given by $x_{k-1} = \cdots = x_N = 0$.  The intersection $Z \coloneqq \{f_2 = \cdots = f_m = 0\} \cap \Pi$ has positive dimension because $k > m$, so choose a point $p \in Z \setminus \{0\} \subset \mathbb{A}^{N+1}$.  Then $f_1$ and all its derivatives are identically zero at $p$, so the affine cone over $X$ is singular at $p$ by the Jacobian criterion.

Hence we conclude that each equation $f_i$ involves at least one of the elementary symmetric polynomials or $V$.  Next, note that for any $1 \leq \alpha \leq m$, it's impossible for a subcollection of $\alpha$ of the equations, say $\{f_1,\ldots, f_{\alpha} \}$, to be contained in an ideal generated by $\alpha - 1$ or fewer elements of the set $\{\sigma_2,\ldots,\sigma_k,V\}$. 
Otherwise the locus where these $\alpha-1$ elements and $f_{\alpha+1}, \ldots, f_m$ are zero would be a subvariety of $\pp$ of codimension at most $m-1$ contained in $X$, but this contradicts the fact that $X$ is codimension $m$.
 
 We'll use the above observation to show by induction on \(\alpha\) that, for any subset of $\alpha$ equations from $\{f_1,\ldots,f_{\alpha}\}$, with $1 \leq \alpha \leq m$, we have $\deg(f_1) + \cdots + \deg(f_{\alpha}) \geq \left(\frac{(\alpha+1)(\alpha+2)}{2}-1\right)b$.  Here $b$ is the weight from part \eqref{part:wci_Ak-standard} of the lemma. In the base case, we've already shown that a single equation $f_1$ must include some polynomial from $\{\sigma_2,\ldots,\sigma_k,V\}$, so it has degree at least $\deg(\sigma_2) = 2b$, since $\sigma_2$ has the smallest degree.  By the inductive hypothesis, suppose that any subset of $\alpha-1$ polynomials from $T = \{f_1,\ldots,f_{\alpha}\}$ satisfies the corresponding inequality on degree.  Some equation from $T$ must be of degree at least $\deg(\sigma_{\alpha+1})$ or else only $\sigma_2,\ldots,\sigma_{\alpha}$ would appear in equations in $T$, contradicting the previous paragraph.  Since the sum of degrees of the other $\alpha-1$ equations is at least $\deg(\sigma_2) + \cdots + \deg(\sigma_{\alpha})$, this completes the induction.  We note that the Vandermonde polynomial has degree $\binom{k}{2}b$, which is larger than the degree of any of the elementary symmetric polynomials;  thus, its degree did not feature in the bounds just proved.  
\end{proof}

The following lemma will nearly finish the proof.

\begin{lemma}\label{lem:total-degree}
    Suppose that the same assumptions from \Cref{lem:large_dim_wci_facts} hold. Then the total degree $d = d_1 + \cdots + d_m$ of $X$ satisfies
    $$d \geq a_{k-1} + \cdots + a_N + \left(\frac{(k-n-1)(k-n)}{2}-1\right)b$$
    where \(b\) is the weight in \Cref{lem:large_dim_wci_facts}.\eqref{part:wci_Ak-standard}.
    If equality holds, then $N = k-2$, so there are no additional weights on the right-hand side, and $\pp(a_0,\ldots,a_N) \cong \pp^{k-2}$.
\end{lemma}

\begin{proof}
Assume that $f_1,\ldots,f_m$ are ordered by increasing degree. By \Cref{lem:large_dim_wci_facts}.\eqref{part:wci-Sk-total-deg}, the first $k-n-2$ equations satisfy $\deg(f_1) + \cdots + \deg(f_{k-n-2}) \geq \deg(\sigma_2) + \cdots + \deg(\sigma_{k-n-1}) = \left(\frac{(k-n-1)(k-n)}{2}-1\right)b$.

If \(m>k-n-2\), then we may apply \Cref{IF_lemma}.\eqref{part:IF_lemma-1} for \(0\leq j\leq N-k+1\) to obtain
\begin{equation}\label{eqn:deg-lastN+2-k}
    \deg(f_{k-n-1})+ \cdots + \deg(f_m) > a_{k-1} + \cdots + a_N.
\end{equation}
(Contrary to the notation of that lemma, the weights $a_{k-1}, \ldots, a_N$ might not be the largest of the $a_i$, but the same inequality will certainly also hold for a different subset of weights with smaller total.)  Here we note that $m \geq N-k+2$ because $k \geq n+2 = N-m+2$.

Adding the two inequalities together yields the inequality in the statement of the lemma, and we see that equality can only occur when there is no contribution from \eqref{eqn:deg-lastN+2-k}. This only occurs when \(m=k-n-2\), so that $N = k-2$ and all the weights are the same.  Since our weighted projective space is well-formed, this implies $b = 1$ and $\pp(a_0,\ldots,a_N) \cong \pp^{k-2}$.
\end{proof}

We can now conclude the proof of \Cref{thm:max_symmetric_action}.

\begin{proof}[Proof of \Cref{thm:max_symmetric_action}]
In light of \Cref{prop:low_dim_wci}, we may assume that $n \geq 4$, for $X$ Fano, and $n \geq 3$, for $X$ Calabi--Yau. By \Cref{lem:no-linear-term}, we may exclusively consider well-formed $X$ with the property that no defining equation has a linear term.  Finally, we may also assume that the $S_k$-action satisfies $k \geq c_{\mathrm{Fano}}(n)$ in the Fano case, or $k > c_{\mathrm{CY}}(n)$ in the Calabi--Yau case.  Indeed, if these inequalities on $k$ are not satisfied, then the conclusion of \Cref{thm:max_symmetric_action} automatically holds.  In summary, we have reduced to the setting where the conditions of \Cref{lem:large_dim_wci_facts} and \Cref{lem:total-degree} are satisfied, so we may apply the conclusions of these lemmata.

For $X$ to be Fano (resp. Calabi--Yau), we must have $d < a_0 + \cdots + a_N = (k-1)b + a_{k-1} + \cdots + a_N$ (resp. $\leq$).  This inequality together with \Cref{lem:total-degree} implies 
\begin{equation}
\label{degree_condition}
\frac{(k-n-1)(k-n)}{2} < k \text{ (resp. $\leq$)}.
\end{equation}
For a fixed \(n\), \eqref{degree_condition} with a strict inequality holds for an integer \(k\) if and only if \(n+1+\frac{1-\sqrt{8n+9}}{2} < k < n+1+\frac{1+\sqrt{8n+9}}{2}\).
Therefore, \(k\leq n+\left\lceil\frac{1+\sqrt{8n+9}}{2}\right\rceil = c_{\mathrm{Fano}} (n)\).  \Cref{ex:optimal_example} shows that the bound for Fano $X$ is sharp for all $n \geq 1$.

Similarly, in the Calabi--Yau case, \eqref{degree_condition} with a non-strict inequality holds for an integer \(k\) if and only if \(n+1+\frac{1-\sqrt{8n+9}}{2} \leq k \leq n+1+\frac{1+\sqrt{8n+9}}{2}\).  Therefore, \(k\leq n+\left\lfloor\frac{1+\sqrt{8n+9}}{2}\right\rfloor  + 1 = c_{\mathrm{CY}}(n)\).  This completes the proof.
\end{proof}

\begin{proof}[Proof of \Cref{introthm-bound-Fano-WCI}]
It follows from \Cref{thm:max_symmetric_action} that the largest symmetric group action on a weighted complete intersection is by $c_{\mathrm{Fano}}(n) = n + \left\lceil \frac{1 + \sqrt{8n + 9}}{2} \right\rceil$.  Taking the limit of $c_{\mathrm{Fano}}(n)/(n+1)$ as $n \rightarrow \infty$ gives the required result.
\end{proof}

\section{Maximally symmetric varieties}
\label{sect:max_sym_vars}

In this section, we study 
maximally symmetric Fano weighted complete intersections.  These are the Fano weighted complete intersections of dimension $n$ which have a faithful action by $S_k$, where $k = c_{\mathrm{Fano}}(n)$ is the largest possible.

Using the setup of \Cref{sect:wci}, we can further limit the possible behavior of maximally symmetric Fano weighted complete intersections.

\begin{proposition}
\label{prop:zero_or_one_extra_weights}
Suppose that $X$ is a maximally symmetric Fano weighted complete intersection of dimension $n \geq 4$ with action by $S_k$, i.e., $k = c_{\mathrm{Fano}}(n)$.  Suppose further that $X$ is quasismooth and well-formed and no defining equation contains a linear term.  Then, $X$ is embedded in either $\pp^{k-2}$ or $\pp^{k-1}(1^{(k-1)},a)$.
\end{proposition}

As before, $X$ will denote a quasismooth weighted complete intersection $X_{d_1,\ldots,d_m} \subset \pp(a_0,\ldots,a_N)$ defined by equations $f_1,\ldots,f_m$ which are weighted homogeneous of degrees $d_1,\ldots,d_m$, respectively.

\begin{proof}
The conditions of \Cref{lem:large_dim_wci_facts} are met, so the three properties listed there hold for $X$.  We retain the notation from that lemma.

In \Cref{lem:total-degree}, we applied \Cref{lem:large_dim_wci_facts}.\eqref{part:wci-Sk-total-deg} with $\alpha = k-n-2$ and \Cref{IF_lemma}.\eqref{part:IF_lemma-1} to obtain a lower bound on the total degree $d = d_1 + \cdots + d_m$ of the complete intersection $X$.  Now, we'll do nearly the same thing with a different value of $\alpha$ to obtain another useful bound on $d$.  From now on, order the equations $f_1,\ldots,f_m$ by increasing degree.

If the dimension of the ambient weighted projective space $\pp$ is $N = k-2$, then $\pp \cong \pp^{k-2}$ by \Cref{lem:large_dim_wci_facts}.\eqref{part:wci_Ak-standard}.  Otherwise, there is at least one weight \textit{not} contained in the faithful $S_k$-representation, so that $N \geq k-1$.  Since $N = n+m$, this implies $k-n-1 \leq m$, so we may apply \Cref{lem:large_dim_wci_facts}.\eqref{part:wci-Sk-total-deg} to the first $\alpha = k-n-1$ equations $f_1, \ldots, f_{k-n-1}$ to obtain:
$$d_1 + \cdots + d_{k-n-1} \geq \left(\frac{(k-n)(k+1-n)}{2}-1 \right)b.$$
Since $X$ is Fano, the total degree is less than the sum of the weights.  That is,
$$a_0 + \cdots + a_N = (k-1)b + a_{k-1} + \cdots + a_N > d_1 + \cdots + d_m \geq \left(\frac{(k-n)(k-n+1)}{2}-1 \right)b + d_{k-n} + \cdots + d_m.$$
Rearranging this expression gives that 
$$\left(k - \frac{(k-n)(k-n+1)}{2}\right)b > (d_{k-n} + \cdots + d_m) - (a_{k-1} + \cdots + a_N).$$
The key point is that $k = c_{\mathrm{Fano}}(n)$ is the largest integer $k$ which satisfies $k - \frac{(k-n-1)(k-n)}{2} > 0$ (see \eqref{degree_condition}).  On the left-hand side, we've replaced $k$ by $k+1$ in the fraction, so the left-hand side must now be nonpositive.  Hence the right-hand side is actually negative, i.e.,
\begin{equation}
\label{eq:degree_ineq_maximal}
a_{k-1} + \cdots + a_N > d_{k-n} + \cdots + d_m.
\end{equation}
Now we'll assume that $N > k-1$, i.e., there are at least \textit{two} weights not contained in the faithful $S_k$-representation, and derive a contradiction. In the inequality~\eqref{eq:degree_ineq_maximal}, there are $N-k+2$ weights on the left and $m-(k-n)+1 = N-k+1$ degrees on the right; in particular, the assumption that \(N>k-1\) means there is a nonzero number of terms on the right-hand side.

Reorder $a_{k-1}, \ldots, a_N$ by increasing size. Recalling that \(k > n+1\), so that \(m-1>m-k+n\), we have $d_m > a_N, d_{m-1} > a_{N-1}, \ldots, d_{k-n} > a_k$ by \Cref{IF_lemma}.\eqref{part:IF_lemma-1}.  We claim that we can improve the first inequality to $d_m \geq a_N + a_{k-1}$.  Indeed, some equation must involve $x_N$, or else the image of the coordinate point $p \coloneqq (0,\ldots,0,1) \in \mathbb{A}^{N+1}$ of $x_N$ is in $X$ and all partial derivatives of all equations vanish there, contradicting quasismoothness.  If $x_N$ ever appears with an exponent of at least $2$, $d_m \geq 2a_N \geq a_N + a_{k-1}$ and we're done.  If not, since there are no linear terms, $x_N$ always appears multiplied by other variables and hence $p \in X$.  We must then have a monomial of the form $x_j x_N$ with $j \neq N$ in some equation, or else once again all equations would have all partial derivatives vanishing at $p$.  But this $j$ cannot be from $0,\ldots,k-2$, because those variables only appear in degree at least $2$ polynomials $\sigma_2,\ldots,\sigma_k, V$.  We conclude that $j \in \{k-1,\ldots,N-1\}$ so the largest degree $d_m$ is at least $a_N + a_{k-1}$.

In summary, $d_m \geq a_N + a_{k-1}, d_{m-1} > a_{N-1}, \ldots, d_{k-n} > a_k$.  Adding these together contradicts inequality \eqref{eq:degree_ineq_maximal}.  We have thus shown that \(k-2\leq N\leq k-1\). That is, the ambient weighted projective space $\pp$ is of the form either $\pp^{k-2}$ or $\pp(1^{(k-1)},a)$, where in the second case we note that $b = 1$ or else $\pp$ is not well-formed.
\end{proof}

\Cref{introthm:maximally-symmetric-Fano-WCI} states that maximally symmetric Fano weighted complete intersections are finite covers of complete intersections in $\pp^N$ cut out by symmetric polynomials. This will now follow quickly from \Cref{prop:zero_or_one_extra_weights}.  We omit the case of dimension $n = 2$ in \Cref{introthm:maximally-symmetric-Fano-WCI} because the largest symmetric group inside $\Cr(2)$ is $S_5$, and there is a copy of $S_5$ contained in $\Cr(2)$ acting regularly on a cubic surface, which is not defined by a symmetric polynomial.  Furthermore, this action is not conjugate to an $S_5$-action on a minimal rational surface \cite[Table 4]{DI09}.

\begin{proof}[Proof of \Cref{introthm:maximally-symmetric-Fano-WCI}]
As usual, we first deal with low-dimensional cases.  When $n = 1$, $\pp^1$ is the only quasismooth Fano weighted complete intersection, so the theorem is trivial.  For $n = 3$, \cite[Proposition 1.1(ii)]{Prokhorov_space_Cremona} shows that any three-dimensional $S_7$-Mori fiber space over a rationally connected base is equivariantly isomorphic to the complete intersection of Fermat hypersurfaces of degrees $1$, $2$, and $3$ in $\pp^3$.  It follows in particular that this is the only maximally symmetric quasismooth Fano weighted complete intersection of dimension $3$.

For $n \geq 4$, we can apply \Cref{prop:zero_or_one_extra_weights}.  We reduce as before to the case where $X$ has no linear terms and is well-formed using \Cref{lem:no-linear-term}.  This shows that $X$ is isomorphic to a weighted complete intersection in $\pp^{k-2}$ or $\pp^{k-1}(1^{(k-1)},a)$.  In either case, \Cref{lem:large_dim_wci_facts}.\eqref{part:wci-Sk-ideal} already showed that the variables $x_0,\ldots,x_{k-1}$ only appear in the equations of $X$ in the form of elementary symmetric polynomials in  $x_0,\ldots,x_{k-1},y$, plus the Vandermonde polynomial. However, no equation $f_i$ may involve the Vandermonde polynomial $V$ or else the degree would be too high to be Fano.  Hence $X$ is defined by equations which are all symmetric in $x_0,\ldots,x_{k-1},y$.

We may now ``add back on" an additional weight equal to $1$ to make the standard representation into the permutation representation; indeed, we saw that the equations for $X$ are combinations of elementary symmetric polynomials of degrees $2,\ldots,k$ in $x_0, \ldots, x_{k-2}, y = -x_0 - \cdots - x_{k-2}$.  Add the variable $y$ of weight $1$ and the extra linear relation $x_0 + \cdots + x_{k-2} + y = 0$ to see the same $X$ as living inside $\pp^{k-1}$ or $\pp^k(1^{(k-1)},a)$, this time defined by invariants of the permutation representation in the first $k$ variables.

If $X \subset \pp^{k-1}$, we can take the finite cover $X \rightarrow X$ to be the identity and we're done.  If $X \subset \pp^k(1^{(k-1)},a)$ consider the restriction to $X$ of the rational map $\pi: \pp^k(1^{(k-1)},a) \dashrightarrow \pp^{k-1}$ forgetting the last weight.  For $X$ to be quasismooth in $\pp^{k-1}(1^{(k-1)},a)$, there must be a monomial of the form $z^r$ appearing in some $f_i$, where $z$ is the variable of weight $a$.  Otherwise, the coordinate point of $z$ would be contained in $X$ and all partial derivatives of all equations would vanish there, since the other variables always appear as part of symmetric polynomials of degree at least $2$.

It follows that the restriction $\pi|_X: X \rightarrow \mathrm{im}(X)$ is a morphism because the only basepoint of $\pi$ is the coordinate point of the last variable, which we saw cannot be contained in $X$.  The image $Y \coloneqq \mathrm{im}(X)$ is clearly defined by symmetric polynomials in $\pp^{k-1}$ and the map has finite fibers, hence is finite.
\end{proof}

Non-trivial finite covers do appear in maximally symmetric examples; see \Cref{ex:one_extra_weight}.

We can say something more precise in the case that a maximally symmetric Fano weighted complete intersection in addition has the largest possible index of $-K_X$. \Cref{introthm:maximally-symmetric-large-index-Fano-WCI} is a direct consequence of the following statement:

\begin{theorem}
\label{thm:maximally_symmetric_largest_index}
Let $X$ be a quasismooth Fano weighted complete intersection of dimension $n \geq 2$ with faithful $S_k$-action, where $k = c_\mathrm{Fano}(n)$ is the upper bound of \Cref{thm:max_symmetric_action}. Then the index $i_X$ of $-K_X$ satisfies:
$$i_X \leq k - \frac{(k-n)(k-n-1)}{2}.$$
When equality holds, $X$ is equivariantly isomorphic to the intersection of Fermat hypersurfaces of degrees $1, \ldots, k-n-1$ in $\pp^{k-1}$.
\end{theorem}

\begin{proof}
As usual, we'll first deal with low dimensions.  By \cite{DI09}, the possible actions of $S_5$ on del Pezzo surfaces are on $\pp^1 \times \pp^1$, the Clebsch diagonal cubic surface, and the degree 5 del Pezzo. Only the first case has the maximal index of $-K_X$ equal to $2$, and this $\pp^1 \times \pp^1$ is the quadric which is a Fano--Fermat complete intersection in $\pp^4$.  In dimension $3$, we know that up to equivariant isomorphism the unique $S_7$-action on a Fano quasimsooth weighted complete intersection is on the $(1,2,3)$-Fano--Fermat complete intersection in $\pp^6$, of index $1$.

Suppose $n \geq 4$. Assume that $X$ is well-formed and has no linear terms in its defining equations.  We may assume this without loss of generality by \Cref{lem:no-linear-term}, which allows us to replace the original $X \subset \mathbb{P}$ with an isomorphic $X' \subset \mathbb{P}'$ with the desired properties.  By taking the $S_k$-action on $X'$ to be the one induced by this isomorphism, we can ensure $X \cong X'$ is equivariant.

By \cite[Theorem 2.15]{PrS2}, the class group of a quasismooth well-formed weighted complete intersection of dimension at least $3$ is isomorphic to $\Z$ with generator $\mathcal{O}_X(1)$.  Therefore, the index of $-K_X$ for a Fano weighted complete intersection equals $a_0 + \cdots + a_N - d$, where $d = d_1 + \cdots + d_m$ of $X$ is the total degree of $X$. We know from \Cref{lem:total-degree} that the total degree  satisfies
$$d \geq a_{k-1} + \cdots + a_N + \left(\frac{(k-n-1)(k-n)}{2}-1\right)b,$$
while the sum of the weights is $a_0 + \cdots + a_N = (k-1)b + a_{k-1} + \cdots + a_N$.  Therefore, 
$$i_X \leq \left(k - \frac{(k-n-1)(k-n)}{2}\right)b.$$
We learned in \Cref{prop:zero_or_one_extra_weights} that $b = 1$ in all maximal examples, so this gives the desired index inequality. 

Suppose now that equality holds. \Cref{lem:total-degree} also showed that this inequality can only be an equality if $X$ is actually a complete intersection in $\pp^{k-2}$ defined by invariants of the standard representation.  The codimension of $X$ is therefore $k-n-2$.  The minimum total degree of $k-n-2$ equations is precisely $\frac{(k-n-1)(k-n)}{2}-1$ and this can only occur when the defining ideal is $\langle \sigma_2,\ldots,\sigma_{k-n-1} \rangle$.  Add back on the extra weight as above (this operation is an $S_k$-equivariant isomorphism on $X$) and note that the elementary symmetric polynomials $\sigma_1,\ldots,\sigma_{\alpha}$ generate the same ideal as the Fermat polynomials $p_1,\ldots,p_{\alpha}$ (over a field).  Therefore, $X$ is $S_k$-equivariantly isomorphic to the Fano--Fermat complete intersection $\{p_1 = \cdots = p_{k-n-1} = 0\} \subset \pp^{k-1}$ of \Cref{ex:optimal_example}.  Note that this $X$ is smooth by \Cref{lem:Fermat_ci_smooth}.
\end{proof}

\begin{proof}[Proof of \Cref{introthm:maximally-symmetric-large-index-Fano-WCI}]
In each dimension $n \geq 2$, \Cref{introthm:maximally-symmetric-large-index-Fano-WCI} directly follows from \Cref{thm:maximally_symmetric_largest_index}. The latter theorem omits the case of $n = 1$ because the index of $-K_{\pp^1}$ is $2$ rather than $1$, as the formula predicts.  Nevertheless, $\pp^1$ is still a Fano--Fermat variety $\{x_0 + x_1 + x_2 + x_3 = x_0^2 + x_1^2 + x_2^2 + x_3^2 = 0\} \subset \pp^3$ with the $S_4$-action by permutation, so \Cref{introthm:maximally-symmetric-large-index-Fano-WCI} holds in all dimensions.
\end{proof}

\section{Symmetries and boundedness}\label{sec:sym-bound}

In this section, we prove statements about the boundedness of Fano $4$-folds 
and $5$-dimensional klt singularities
admitting $S_8$-actions. First, we recall the concepts of dual complexes and coregularity.

\begin{definition}
{\em 
Let $E$ be a simple normal crossing divisor on a smooth variety $X$.
The \defi{dual complex} $\mathcal{D}(E)$ is the CW complex whose vertices correspond to the components of $E$ and whose $k$-cells correspond to the irreducible components of the intersection of $k+1$ components of $E$.

Let $(X,\Gamma)$ be a log Calabi--Yau pair.
Let $\pi\colon Y \rightarrow X$ be a log resolution of $(X,\Gamma)$.
Write $\pi^*(K_X+\Gamma)=K_Y+\Gamma_Y$.
Let $S_Y$ be the sum of all the components of $\Gamma_Y$ that appear with coefficient $1$.
The \defi{dual complex} $\mathcal{D}(Y,\Gamma_Y)$ of $(Y,\Gamma_Y)$ is 
the CW complex $\mathcal{D}(S_Y)$.
}
\end{definition}

In~\cite[Theorem 3]{dFKX17}, the authors show that
the homotopy class of $\mathcal{D}(S_Y)$ is indepenent of the chosen log resolution.
More precisely, given two log resolutions $Y\rightarrow X$ and $Y'\rightarrow X$ of $(X,\Gamma)$, the dual complexes $\mathcal{D}(S_Y)$ and $\mathcal{D}(S_{Y'})$ are simple homotopy equivalent to each other.
Thus, we have a well-defined \defi{dual complex} of $\mathcal{D}(X,\Gamma)$.
If $G$ is a finite group acting on $(X,\Gamma)$, we may consider a $G$-equivariant log resolution of the pair.
Hence, $G$ acts on $\mathcal{D}(X,\Gamma)$.
The dual complex of a log Calabi--Yau pair is in general a pseudo-manifold~\cite[Theorem 1.6]{FS20};
however, in dimension at most $4$, we know that they are orbifolds~\cite[Proposition 5]{KX16}.

\begin{definition}
{\em 
Let $(X,\Gamma)$ be a log Calabi--Yau pair. 
The \defi{coregularity} of $(X,\Gamma)$, written ${\rm coreg}(X,\Gamma)$, is defined to be 
\[
\dim X - \dim \mathcal{D}(X,\Gamma)-1.
\]
Let $(X,B)$ be a log Fano pair.
The \defi{coregularity} of $(X,B)$ is the minimum among the coregularities 
of $(X,\Gamma)$ where $\Gamma\geq B$ and 
$(X,\Gamma)$ is log Calabi--Yau. 
The coregularity of a $n$-dimensional log Fano pair is contained in the set 
$\{0,\dots,n\}$.
}
\end{definition}

The concept of coregularity has recently been connected with log canonical thresholds, indices of Calabi--Yau pairs, and complements of Fano varieties (see~\cite{FMP22,FMM22,FFMP22}).
For log Calabi--Yau pairs, the coregularity is independent of the chosen crepant model:

\begin{lemma}[{\cite[Proposition 3.11]{FMM22}}]\label{lem:coreg-crep-bir}
Let $(X,\Gamma)$ be a log Calabi--Yau pair.
Let $(X',\Gamma')$ be a crepant model of $X$, i.e., a birational log Calabi--Yau pair for which there exists a common resolution $p\colon Y\rightarrow X$ and $q\colon Y\rightarrow X'$ with $p^*(K_X+\Gamma)=q^*(K_{X'}+\Gamma')$.
Then ${\rm coreg}(X,\Gamma)={\rm coreg}(X',\Gamma')$.
\end{lemma}

The following lemma states that the dimension of dual complexes of log Calabi--Yau pairs is preserved under finite quotients.

\begin{lemma}\label{lem:dim-dc-quot}
Let $(X,\Gamma)$ be a log Calabi--Yau pair.
Let $G\leqslant {\rm Aut}(X,\Gamma)$ be a finite group. 
Let $Y\coloneqq X/G$, let $p\colon X\rightarrow Y$ be the quotient morphism, and let $\Gamma_Y$ be the boundary divisor for which $p^*(K_Y+\Gamma_Y)=K_X+\Gamma$.
Then, we have that $\dim \mathcal{D}(Y,\Gamma_Y)=\dim \mathcal{D}(X,\Gamma)$.
\end{lemma}

\begin{proof}
We proceed by induction on the dimension of $X$. The case of dimension $1$ is clear.
By passing to a $G$-equivariant dlt modification, we may assume that $(X,\Gamma)$ is dlt.
By~\cite[Theorem 1.6]{FS20}, the dual complex $\mathcal{D}(X,\Gamma)$ is a equidimensional pseudomanifold.
If $(X,\Gamma)$ is klt, then the statement is clear,
since in this case both $(X,\Gamma)$ and $(Y,\Gamma_Y)$ are klt, so their dual complexes have dimension $-1$.
Thus, we may assume that $\lfloor \Gamma\rfloor$ is non-empty.
Let $S\subset \lfloor \Gamma\rfloor$ be an irreducible component and $S_Y$
the image of $S$ on $Y$.
Let $(S,\Gamma_S)$ be the log Calabi--Yau pair obtained by adjunction of $(X,\Gamma)$ to $S$,
and let $(S_Y,\Gamma_{S_Y})$ be the log Calabi--Yau pair obtained by adjunction of $(Y,\Gamma_Y)$ to $S_Y$.
Note that $\dim \mathcal{D}(X,\Gamma)=\dim(S,\Gamma_S)+1$
and $\dim \mathcal{D}(Y,\Gamma_Y)=\dim(S_Y,\Gamma_{S_Y})+1$.
Indeed, the dual complex $\mathcal{D}(S,\Gamma_S)$ is the link of the vertex $v_S$ corresponding to $S$ in $\mathcal{D}(X,\Gamma)$,
and the analogous statement holds for $(S_Y,\Gamma_{Y_S})$ and $(Y,\Gamma_Y)$.
Let $G_S\leqslant G$ be the subgroup fixing $S$.
Then $G_S$ acts on $(S,\Gamma_S)$.
By construction, we have that $p_S\colon S\rightarrow S_Y$ is the quotient morphism by $G_S$ and $p_s^*(K_{S_Y}+\Gamma_{S_Y})=K_S+\Gamma_S$.
By induction on the dimension, we have that 
$\dim \mathcal{D}(S,\Gamma_S)=\dim \mathcal{D}(S_Y,\Gamma_{S_Y})$.
This finishes the proof.
\end{proof}

Now, we prove the main global statement of this section. To do so, we prove first lemmata regarding 
finite actions on normal varieties, 
alternating group actions on Calabi--Yau surfaces and 
Calabi--Yau $3$-folds, 
and subgroups of the special orthogonal groups.

\begin{lemma}\label{lem:normal+cyclic}
Let \(X\) be a normal variety and \(E\) a prime divisor on \(X\). If \(G\leqslant\Aut(X)\) is a finite subgroup that fixes \(E\) pointwise, then \(G\) is a normal cyclic subgroup of \(\Aut(X,E)\).
\end{lemma}

\begin{proof}
    By~\cite[Corollary 2.13]{BFMS20} we know that \(G\) is cyclic. For normality, let \(h\in\Aut(X,E), g\in G\), and \(x\in E\). Then \(h^{-1}(x)\in E\) so \((hg h^{-1})(x)=(hg)(h^{-1}(x))=h(h^{-1}(x))=x\).
\end{proof}

\begin{lemma}\label{lem:A_8-CY-surface}
    Let $H$ be a finite group and $H\rightarrow A_8$ a surjective group homomorphism. 
    Let $(X,\Gamma)$ be a log canonical Calabi--Yau surface. 
    Then, $(X,\Gamma)$ does not admit a faithful $H$-action.
\end{lemma}

\begin{proof}
We proceed by contradiction.
Let $(X,\Gamma)$ be a log canonical Calabi--Yau surface that admits a faithful action by $H$.
By passing to an $H$-equivariant dlt modification, we may assume that $(X,\Gamma)$ is dlt. 

First, assume that $\Gamma\neq 0$. 
Then, we may run an $H$-equivariant $K_X$-MMP that terminates with a Mori fiber space.
Let $X\rightarrow X_1\rightarrow \cdots \rightarrow X_k$ be the steps of this MMP and $X_k\rightarrow Z$ be the Mori fiber space. 
Let $\Gamma_k$ be the push-forward of $\Gamma$ on $X_k$. 
First, assume that $Z$ is a point.
So $X_k$ is a Fano surface. 
Let $N$ be the kernel of the homomorphism $H\rightarrow A_8$.
The quotient $Y\coloneqq X_k/N$ is a Fano type surface that admits a $A_8$-action. 
In particular, an $A_8$-equivariant resolution of $Y$ is a smooth rational surface with a faithful $A_8$-action. 
This is impossible, as the plane Cremona group does not admit a subgroup isomorphic to $A_8$ (see~\cite{DI09}).
Now, assume that $Z$ is a curve. We have a short exact sequence
\[\xymatrix{
1 \ar[r] &  G_F \ar[r] &  H \ar[r] &  G_Z \ar[r] & 1,}
\]
where $G_F$ acts on the general fiber of $X_k\rightarrow Z$
and $G_Z$ acts on $Z$.
By the canonical bundle formula, the curve $Z$ has genus either $0$ or $1$.
Note that either $G_F$ or $G_Z$ admits a surjective homomorphism to $A_8$.
Thus, we get a faithful action on a curve by a group $G$ that surjects onto $A_8$.
By taking the quotient by the kernel of $G\rightarrow A_8$, we obtain a curve of genus at most $1$
that admits a faithful $A_8$-action.
This is impossible due to the classification of finite subgroups of $\PGL_2(\cc)$ and the classification of finite actions on genus $1$ curves; indeed, the automorphism group of a genus $1$ curve $C$ is a semidirect product of an abelian translation group with a cyclic group of order $2,4,$ or $6$ (see also the proof of \Cref{prop:low_dim_wci}).

Now, assume that $\Gamma=0$.
Then $X$ is a klt Calabi--Yau surface.
Let $Y\rightarrow X$ be a $H$-equivariant resolution
and $\varphi^*(K_X)=K_Y+D_Y$.
If $D_Y\neq 0$, then we proceed as in the previous paragraph. 
Thus, we may assume that $X$ has canonical singularities and $Y$ is a smooth surface with $K_Y\sim_\qq 0$.
Let $\tilde{Y}\rightarrow Y$ be the universal cover of $Y$. Then, there is a finite group $H_Y$ acting on $\tilde{Y}$ that surjects onto $A_8$.
By construction $\tilde{Y}$ is either an abelian surface or a K3 surface.
First, assume that $\tilde{Y}$ is a K3 surface.
Then, as in the proof of \Cref{prop:low_dim_wci}, we have an exact sequence 
\[\xymatrix{
1 \ar[r] &  H_Y^{\rm symp}
\ar[r] &  H_Y \ar[r] &  
\zz/m \ar[r] &  1, }
\]
where $H_Y^{\rm symp}$ is a finite group acting by symplectic automorphisms on the K3 surface
$\tilde{Y}$.
We conclude that 
$H_Y^{\rm symp}$ surjects onto $A_8$.
In particular $|H_Y^{\rm symp}|\geq 8!/2$.
This leads to a contradiction by the classification of finite groups acting sympletically on K3 surfaces (see~\cite{Muk88}).
Now, assume that $\tilde{Y}$ is an abelian surface and so $\tilde{Y}\rightarrow Y$ is an isomorphism.
Let $T_Y\leqslant \Aut(Y)$ be the group of translations
of the abelian surface.
Then, we have an exact sequence
\[\xymatrix{
1\ar[r] &  H_Y\cap T_Y \ar[r] &  H_Y \ar[r] & G_Y \ar[r] & 1}.
\]
Since $T_Y$ is abelian, we conclude that $H_Y\cap T_Y$ does not surject onto $A_8$.
So, $G_Y$ must surject onto $A_8$.
Observe that $G_Y$ is a group of automorphisms of the abelian surface
that fixes the identity and surjects onto $A_8$.
In this case, we get a contradiction
by the classification
of finite groups acting on abelian surfaces (see~\cite{Fuj88}).
\end{proof}

\begin{lemma}\label{lem:A8-CY-3fold}
    Let $H$ be a finite group and $H\rightarrow A_8$ a surjective group homomorphism.
    Let $(X,\Gamma)$ be a log canonical Calabi--Yau $3$-fold with $\Gamma\neq 0$. 
    Then, $(X,\Gamma)$ does not admit a faithful $H$-action.
\end{lemma}

\begin{proof}
By means of contradiction, assume that a Calabi--Yau $3$-fold $(X,\Gamma)$ with a faithful $A_8$-action exists.
We run a $H$-equivariant $K_X$-MMP.
Since $\Gamma\neq 0$, this MMP must terminate with a Mori fiber space $X_k\rightarrow Z$.
By replacing $X$ with $X_k$ 
and $\Gamma$ with its push-forward on $X_k$, we may assume that $X$ itself admits a Mori fiber space $X\rightarrow Z$.
If $Z$ is a point, then $X$ is a Fano variety. But $A_8$ does not act faithfully on a rationally connected $3$-fold~\cite{BCDP-23}, so we get a contradiction.
Assume that $Z$ is positive-dimensional.
We have a short exact sequence
\[\xymatrix{
1 \ar[r] &  G_F \ar[r] & H \ar[r] & G_Z\ar[r] &  1,}
\]
where $G_F$ acts on the general fiber of $X_k\rightarrow Z$ and $G_Z$ acts on the base $Z$.
By an equivariant version of the canonical bundle formula (see~\cite[Lemma 2.32]{Mor21}), we obtain a $G_Z$-equivariant boundary $B_Z$ such that 
$(Z,B_Z)$ is Calabi--Yau and log canonical. Note that either $G_F$ or $G_Z$ admits a surjective homomorphism onto $A_8$.
In either case, we get a group $G$ surjecting onto $A_8$ and acting on a log Calabi--Yau pair of dimension at most $2$. This contradicts~\Cref{lem:A_8-CY-surface}.
\end{proof}

\begin{lemma}\label{lem:finite-subgroup-SO}
Let $G$ be a finite subgroup of ${\on O}(k)$ for $k\leq 4$. Then $G$ does not admit a surjective homomorphism to $A_8$.
\end{lemma}

\begin{proof}
It is enough to consider finite subgroups of $\SO(k)$ for $k\leq 4$. 
The statement is clear for $k\leq 3$.
Indeed, a finite subgroup of $\SO(k)$ with $k\leq 3$ is 
cyclic, dihedral, icosahedral, tetrahedral, or octahedral.
For \(k=4\), recall that we have a short exact sequence
\[\xymatrix{
1\ar[r] &  \zz/2 \ar[r] &  \SO(4)\ar[r] &  \SO(3)\times \SO(3)\ar[r] &  1. }
\]
Thus, if there is a finite subgroup of ${\on O}(4)$ that surjects onto $A_8$, then there
is a finite subgroup of $\SO(3)$
that surjects onto $A_8$.
This leads to a contradiction.
\end{proof}

Now, we are ready to prove the boundedness of $S_8$-equivariant Fano $4$-folds. In what follows, we show a version of \Cref{introthm:S8-4-fold-bounded} for log pairs. This version for log pairs will be used to prove \Cref{introthm:S8-5-dim-klt}.

\begin{theorem}\label{thm:S8-4-log-pair-bounded}
Let $\mathcal{B}\subset [0,1]$ be a set satisfying the DCC and $\overline{\mathcal{B}}\subset \qq$.
Let $\mathcal{F}_{4,8,\mathcal{B}}$ be the class 
of $4$-dimensional $S_8$-equivariant klt pairs $(X,B)$
for which $-(K_X+B)$ is ample
and ${\rm coeff}(B)\subset \mathcal{B}$.
Then the class $\mathcal{F}_{4,8,\mathcal{B}}$ is log bounded.
\end{theorem}

\begin{proof}[Proof of \Cref{thm:S8-4-log-pair-bounded}]
We will show that the class of \(S_8\)-equivariant klt log Fano \(4\)-dimensional pairs is log bounded. 

Let $(X,B)$ be a klt \(S_8\)-equivariant log Fano \(4\)-dimensional pair with ${\rm coeff}(B)\subset \mathcal{B}$, and let \(\pi\colon X\to Y\) be the quotient. By Riemann--Hurwitz we can write \[\pi^*(K_Y+B_Y)=K_X+B\] where \(B_Y\) is an effective divisor.
By~\cite[Lemma 5.2]{FM20}, there exists a set $\mathcal{C}\subset [0,1]$ satisfying the DCC and $\overline{\mathcal{C}}\subset \mathbb{Q}$ 
such that ${\rm coeff}(B_Y)\in \mathcal{C}$.
The set $\mathcal{C}$ only depends on $\mathcal{B}$; hence, it is independent of the chosen pair $(X,B)$.
So \((Y,B_Y)\) is a klt log Fano pair. Indeed, a pair is klt if and only if a finite pullback of it is klt, see~\cite[Proposition 2.11]{Mor20c}. We proceed in three cases, depending on \(\coreg(Y,B_Y)\). \\

\textit{Case 1}: In this case we assume that the coregularity of the pair \((Y,B_Y)\) is \(4\). 

In this case, every log Calabi--Yau structure \((Y,\Gamma_Y)\) with \(\Gamma_Y\geq B_Y\) satisfies that \((Y,\Gamma_Y)\) is klt. Hence \((Y,B_Y)\) is an exceptional Fano pair~\cite[Section 2.15]{Bir19}. By Birkar's boundedness of exceptional Fano pairs~\cite[Theorem 1.11]{Bir19}, we conclude that \((Y,B_Y)\) is log bounded. That is, there exist constants \(k_0,k_1\) such that for any \((Y,B_Y)\) log Fano klt pair of dimension $4$ and coregularity \(4\), there is a very ample line bundle \(\mathcal A_Y\) on \(Y\) with \[ \mathcal A_Y^4 \leq k_0 \quad \text{ and } \quad \mathcal A_Y^3\cdot B_Y\leq k_1.\] Choose \(A_Y\in |\mathcal A_Y|\) with no components in the branch locus of \(\pi\). Then \(\pi^* A_Y \coloneqq A_X\) satisfies \[A_X^4=(8!)^4 A_Y^4 \leq (8!)^4 k_0,\] so \(X\) is bounded.
In particular, we have $\mathcal{A}_X^3 \cdot -K_X\leq k_2$
for some constant $k_2$ independent of $X$.
On the other hand, note that 
\[
\mathcal{A}_X^3 \cdot (K_X+B) \leq 0, 
\]
so 
\[
\mathcal{A}_X^3 \cdot B \leq \mathcal{A}_X^3 \cdot -K_X\leq k_2.
\]
Since the coefficients of $B$ are bounded below, we conclude that every component of $B$ has degree bounded above with respect to $\mathcal{A}_X$.
Thus, we conclude that the pairs $(X,B)$ are log bounded.\\

\textit{Case 2}: In this case we assume that the coregularity of the pair \((Y,B_Y)\) is \(3\).

By~\cite[Lemma 2.18]{Mor21d}
and~\cite[Theorem 1.2]{FM20}, there exists a constant \(N\) such that the following holds: For any klt log Fano pair \((Y,B_Y)\) with coregularity \(3\), there exists \(\Gamma_Y \geq B_Y\) such that: \begin{itemize} \item \((Y,\Gamma_Y)\) is log canonical, \item \(\mathcal D(Y,\Gamma_Y)\) is zero-dimensional, and \item \(N(K_Y+\Gamma_Y)\sim 0\). \end{itemize} Let \((X,\Gamma)\) be the log Calabi--Yau pair defined by \[K_X+\Gamma = \pi^*(K_Y+\Gamma_Y).\] Then the following hold: \begin{itemize}
    \item \(S_8\leqslant\Aut(X,\Gamma)\),
    \item \(\mathcal D(X,\Gamma)\) is zero-dimensional (by~\Cref{lem:dim-dc-quot}), and
    \item \(N(K_X+\Gamma)\sim 0\).
\end{itemize}
By \cite[Theorem 1.6]{FS20}, \(\mathcal D(X,\Gamma)\) is either one point or two points. First assume \(\mathcal D(X,\Gamma)\) is two points. Let \((X',\Gamma')\to(X,\Gamma)\) be an \(S_8\)-equivariant dlt modification and \(E_0,E_1\subset\lfloor\Gamma'\rfloor\) the two components. Note that \(E_0\) and \(E_1\) are each \(A_8\)-invariant. So we may run an \(A_8\)-equivariant \((K_{X'}+\Gamma'-E_0-E_1)\)-MMP, which terminates with a Mori fiber space, because \(K_{X'}+\Gamma'-E_0-E_1\) is not pseudo-effective. Let
\[\xymatrix{ X' \ar@{-->}[r] & X_1 \ar@{-->}[r] & X_2 \ar@{-->}[r] & \cdots \ar@{-->}[r] & X_k \ar[d] \\ & & & & Z}\]
be the steps of the MMP and \(X_k\to Z\) the \(A_8\)-equivariant Mori fiber space. Let \(E_{0,k}\) and \(E_{1,k}\) be the pushforwards of \(E_0\) and \(E_1\), respectively, on \(X_k\). Since \(E_{0,k}+E_{1,k}\) is ample over \(Z\), either \(E_{0,k}\) or \(E_{1,k}\) is horizontal over \(Z\). Furthermore, they have trivial intersection by the assumption on \(\mathcal D(X,\Gamma)\).
Thus, both divisors $E_{0,k}$ and $E_{1,k}$ must be horizontal over the base, 
otherwise they would have non-trivial intersection.
Since \(\rho^{A_8}(X_k/Z)=1\) and both \(E_{0,k}\) and \(E_{1,k}\) are \(A_8\)-invariant, we conclude that \(E_{0,k}\) and \(E_{1,k}\) are each ample over \(Z\). Then a general fiber \(F\) of \(X_k\to Z\) has dimension \(1\). Indeed, if a general fiber \(F\) of \(X_k\to Z\) has dimension \(\geq 2\), then \({E_{0,k}}|_F\) and \({E_{1,k}}|_F\) intersect non-trivially, leading to a contradiction. Hence, \(\dim Z = 3\) and \(Z\) is rationally connected (as it is the image of the rationally connected variety \(X_k\)). Moreover, letting \(\Gamma_k\) denote the pushforward of \(\Gamma'\), the general fiber of \((X_k,\Gamma_k)\to Z\) is isomorphic to \((\mathbb P^1,\{0\}+\{\infty\})\). So we have an exact sequence \[\xymatrix{1 \ar[r] & G_F \ar[r] & A_8 \ar[r] & G_Z \ar[r] & 1}\] where \(G_F\) acts on the log general fiber \((\mathbb P^1,\{0\}+\{\infty\})\) and \(G_Z\) acts on \(Z\). As \(A_8\) is simple, we have that \(G_Z\) is either trivial or \(A_8\). The latter case does not happen by \cite{BCDP-23}. In the former case, we must have that \(G_F\cong A_8\); however, \(\Aut(\mathbb P^1,\{0\}+\{\infty\})\) is an extension of \(\mathbb G_m\) and \(\mathbb Z/2\), which does not admit an embedding of \(A_8\). Thus, we obtain a contradiction.

Now \(\mathcal D(X,\Gamma)\) is a single point. Let \(\pi\colon (X',\Gamma')\to(X,\Gamma)\) be an $S_8$-equivariant dlt modification. The divisor \(E \coloneqq \lfloor\Gamma'\rfloor\) is fixed by \(S_8\). We proceed in two cases, depending on whether or not \(E=\Gamma'\).

If \(E\neq \Gamma'\), write \(\Gamma'=E+F\) with \(F>0\), and run an \(S_8\)-equivariant \((K_{X'}+\Gamma'-E)\)-MMP. Call the steps \[\xymatrix{ X' \ar@{-->}[r] & X_1 \ar@{-->}[r] & X_2 \ar@{-->}[r] & \cdots \ar@{-->}[r] & X_k \ar[d]^-{\pi} \\  & & & & Z }\]
and denote by \(E_i,\Gamma_i, F_i\) the pushforwards.
Here, $\pi\colon X_k\rightarrow Z$ is the equivariant Mori fiber space. 
If $\dim Z=1$, then we get a contradiction by analyzing the action
on the general fiber, which is a Fano $3$-fold, and the base, which is a rational curve.
Thus, we assume that $\dim Z \geq 2$ or $\dim Z=0$.
In this case, since $E_k$ is ample over $Z$, we conclude that $E_k$ intersects every irreducible divisor on $X_k$.
Indeed, if the irreducible divisor is vertical over $Z$, then $E_k$ 
intersects positively every curve contained in such divisor.
On the other hand, if the irreducible divisor is horizontal over $Z$, 
then it intersects $E_k$ on the general fiber.
However, note that every divisor that is contracted by this MMP
must intersect the strict transform of $E$ positively.
Indeed, every curve that is contracted on this MMP is $(-E)$-negative.
Thus,
if the strict transform of $F$ on $X_k$ is trivial, then
$F_j$ and $E_j$ intersect for some $j<k$.
On the other hand, if the strict transform of $F$ on $X_k$ is non-trivial,
then $F_k$ and $E_k$ intersect.
Thus, for some \(1\leq j \leq k\), we have that \(E_j\) intersects \(F_j\).  Let \((E_j,\Gamma_{E_j})\) be the pair obtained by adjunction of \((X_j,\Gamma_j)\) to \(E_j\). As \(F_j\cap E_j\neq\emptyset\), we have that \(\Gamma_{E_j}\neq 0\). The kernel of the action of \(S_8\) on \(E_j\) is normal and cyclic by \Cref{lem:normal+cyclic}, hence trivial. So \(S_8\) acts faithfully on the 3-dimensional log Calabi--Yau pair \((E_j,\Gamma_{E_j})\) with \(\Gamma_{E_j}\neq 0\). This contradicts~\Cref{lem:A8-CY-3fold}.

It remains to show the case \(\Gamma'=E\). Note that \(\Gamma \neq 0\). If the dlt modification \(\pi\) is non-trivial, then \(\Gamma'=E+\pi_*^{-1}\Gamma\), which is a contradiction. So \(\pi\) is trivial, and we have \(\Gamma=E\) and the pair \((X,\Gamma)\) is dlt. In particular, \((X,\Gamma)\) is plt and, by construction, \(N(K_X+\Gamma)\sim 0\). If \(a_F(X)<\frac{1}{N}\) for some \(F\neq\Gamma\) over \(X\), then \(a_F(X,\Gamma)=0\), which contradicts that \((X,\Gamma)\) is plt. We conclude that \(X\) is \(\frac{1}{N}\)-lc and Fano, and hence bounded by~\cite[Theorem 1.1]{Bir21}.
Then, the boundedness of the pair $(X,B)$ follows as in the first step.\\

\textit{Case 3}: In this case we assume that the coregularity of the pair \((Y,B_Y)\) is \(\leq 2\).

We will show that this case does not happen. In this case, we know there exists \(\Gamma_Y > B_Y \geq 0\) with \(K_Y+\Gamma_Y\sim_{\mathbb Q} 0\) and \(3\geq\dim\mathcal D(Y,\Gamma_Y)\geq 1\). Let \((X,\Gamma)\) be the log pullback of \((Y,\Gamma_Y)\) to \(X\).
We are in the setting of~\Cref{lem:dim-dc-quot}, so we have that \(\dim\mathcal D(X,\Gamma)\in\{1,2,3\}\) and \(K_X+\Gamma\sim_{\mathbb Q} 0\).

Let \((X',\Gamma')\to(X,\Gamma)\) be an \(S_8\)-equivariant dlt modification. The profinite completion \(\hat{\pi}_1(\mathcal D(X',\Gamma'))\) corresponds to a quasi-\'etale cover \((Z',\Gamma_{Z'})\to (X',\Gamma')\) such that \(\mathcal D(Z',\Gamma_{Z'})\) is PL-homeomorphic to either a sphere \(\mathbb S^k\) or a disk \(\mathbb D^k\) with \(k\leq 3\) \cite[Theorem 2 and Paragraph 33]{KX16}.
Since \((Z',\Gamma_{Z'})\to (X',\Gamma')\) is associated to the universal cover of \(\mathcal D(X',\Gamma')\), there is a finite group \(G\) surjecting onto \(S_8\) and acting on \(\mathcal D(Z',\Gamma_{Z'})\). If \(\mathcal D(Z',\Gamma_{Z'})\) is a disk, then \(G\) acts on \(\partial\mathcal D(Z',\Gamma_{Z'}) \cong_{\mathrm{PL}}\mathbb S^{k-1}\) with \(k\leq 3\). 
So in either case, \(G\) acts on a triangulation of a sphere $\mathbb{S}^k$ with $k\leq 3$.
In particular $G$ acts continuously on $\mathbb{S}^k$ with $k\leq 3$.
By~\cite[Theorem 1.1]{Par21}, there is a smooth faithful action of $G$ on $\mathbb{S}^k$ with $k\leq 3$.
Every finite smooth action on a sphere of dimension at most $3$ is conjugate to an orthogonal action (see, e.g., \cite[page 1]{zim18}).
 Hence, we have a homomorphism 
 \[
 \xymatrix{G \ar[r] & {\on O}(k)}
 \] 
 with \(k\leq 4\).
Let \(H\) denote the kernel. 
By~\Cref{lem:finite-subgroup-SO}, 
we conclude that $H$ surjects onto $A_8$.
So \(H\) acts trivially on either \(\mathcal D(Z',\Gamma_{Z'})\) or its boundary, so in particular the \(H\)-action on \(\mathcal D(Z',\Gamma_{Z'})\) has a fixed vertex \(v\in \mathcal D(Z',\Gamma_{Z'})\). Let \(E_v\) be the corresponding divisor on \(\lfloor\Gamma_{Z'}\rfloor\). Then \(E_v\) is fixed by every element of \(H\). By \Cref{lem:normal+cyclic}, the subgroup of \(H\) that fixes \(E_v\) pointwise is normal and cyclic and hence trivial. Hence \(H\) acts faithfully on \(E_v\). Let \((E_v,\Gamma_v)\) be the pair obtained by adjunction of \((Z',\Gamma_{Z'})\) to \(E_v\). Since \(\dim\mathcal D(Z',\Gamma_{Z'})\geq 1\), we have \(\Gamma_v\neq 0\). So \(H\) acts faithfully on a log Calabi--Yau threefold \((E_v,\Gamma_v)\) with \(\Gamma_v\neq 0\). This contradicts~\Cref{lem:A8-CY-3fold}.
\end{proof}

Now, we turn to give a proof of the boundedness up to degeneration
of $5$-dimensional $S_8$-equivariant klt singularities.
The global-to-local argument used in the proof of \Cref{introthm:S8-5-dim-klt}
is very similar to that of \Cref{introthm:quadratic-bound-klt}.

\begin{proof}[Proof of \Cref{introthm:S8-5-dim-klt}]
Let $\mathcal{K}_{5,8,\epsilon}$ be the class of $5$-dimensional $S_8$-equivariant klt 
singularities $(X;x)$ with $\mld(X;x)>\epsilon$.
We show that the class $\mathcal{K}_{5,8,\epsilon}$ is bounded up to degeneration.

Let $(X;x)$ be an element of $\mathcal{K}_{5,8,\epsilon}$.
Let $\pi\colon (X;x) \to (Y;y)$ be the quotient of $(X;x)$ by the $S_8$-action.
Then, there is a boundary $B_Y$ with standard coefficients
for which $(Y,B_Y;y)$ is klt and 
$\pi^*(K_Y+B_Y)=K_X$.
By~\cite[Lemma 1]{Xu14}, there exists a blow-up $\varphi_Y\colon Y'\rightarrow Y$ 
that extracts a unique prime divisor $E'$ that maps to $y\in Y$ and satisfies the following:
\begin{itemize}
    \item the pair $(Y',E'+{\varphi_Y}^{-1}_*B_Y)$ has plt singularities, and 
    \item the divisor $-(K_{Y'}+E'+{\varphi_Y}^{-1}_*B_Y)$ is ample over $Y$.
\end{itemize}
Let $X'\rightarrow X$ be the projective birational morphism obtained by fiber product.
Then $X'$ admits a $S_8$-action and its quotient is $Y'$.
Let $\pi'\colon X'\rightarrow Y'$ be the corresponding quotient map.
Let $K_{X'}+F = {\pi'}^*(K_{Y'}+E'+{\varphi_Y}^{-1}_*B_Y)$.
Since $(X',F)$ is the finite pull-back of a plt pair, we conclude that it is itself plt.
By connectedness of log canonical centers, we conclude that $F$ is prime. 
Thus, $F$ is the unique prime divisor that maps to $x\in X$.
Note that $-(K_{X'}+F)$ is ample over $X$. 
On the other hand, the pair $(X',F)$ admits a faithful $S_8$-action.
By~\Cref{lem:normal+cyclic}, we conclude that $F$ admits a faithful $S_8$-action.
Let $(F,B_F)$ be the log pair obtained by adjunction of $(X',F)$ to $F$.
By construction, the following conditions hold:
\begin{itemize} 
\item $F$ is $4$-dimensional, 
\item $(F,B_F)$ is klt,
\item $-(K_F+B_F)$ is ample, and 
\item $B_F$ has standard coefficients.
\end{itemize} 
By \Cref{thm:S8-4-log-pair-bounded}, we conclude that $(F,B_F)$ belongs
to the log bounded class $\mathcal{F}_{4,8}$.
Then, the class $\mathcal{K}_{5,8,\epsilon}$ is bounded up to degeneration by~\cite[Theorem 1.1]{HLM20}.
We give more details for the benefit of the reader:
we may degenerate the singularity $(X;x)$ to the orbifold cone of $F$ with respect to the $\qq$-polarization $-F|_F$. 
The degree of this $\qq$-polarization is bounded above if the mld of $(X;x)$ is bounded below. 
The central fiber of this degeneration is a cone singularity that belongs to a bounded family by~\cite[Theorem 1]{Mor18c}.
\end{proof}

We finish this section by proving a {\em birational} boundedness statement
for maximally symmetric Fano varietes.
The following theorem states that maximally symmetric Fano varieties are birationally bounded provided some hypothesis that is supported by \Cref{thm:max_symmetric_action}.
Observe that birationally boundedness is much weaker than boundedness.
For example, there are countably many toric Fano varieties of dimension $n$ for $n\geq 2$;
however, the class of toric Fano varieties of dimension $n$ is birationally bounded, 
as each of these varieties is birational to $\pp^n$.

\begin{theorem}
\label{thm:birat_boundedness}
Let $m(n)$ be the maximum integer
for which $S_{m(n)}$ acts faithfully on a $n$-dimensional Fano variety.
Let $\ell(d)$ be the maximum integer 
for which $A_{\ell(d)}$ acts faithfully on a $d$-dimensional Fano variety.
Assume that $m(n)>\ell(d)$ for every $d\leq n-1$.
Then, the class of maximally symmetric
$n$-dimensional Fano varieties 
is birationally bounded.
\end{theorem}

\begin{proof}
Let $X$ be a maximally symmetric $n$-dimensional Fano variety. 
Let $S_m$ be the symmetric group acting on $X$.
Let $X'\rightarrow X$ be an equivariant resolution of singularities.
The Fano variety $X$ is rationally connected, so 
$X'$ is rationally connected.
We run an $S_m$-equivariant minimal model program
\[\xymatrix{
X' \ar@{-->}[r] & X_1' \ar@{-->}[r] & X_2' \ar@{-->}[r] &\dots \ar@{-->}[r] & X'_k}
\]
for $K_{X'}$. 
Since $X'$ is rationally connected and smooth, then $K_{X'}$ is not pseudo-effective, so we have an equivariant Mori fiber space
$X'_k\rightarrow Z$.
To show the result, it suffices to show that $\dim Z=0$, since then $X'_k$ is a terminal $n$-dimensional Fano variety of Picard rank one, so it belongs to a bounded family by~\cite[Theorem 1.1]{Bir21}.

To show that \(\dim Z=0\), assume by contradiction that $\dim Z \geq 1$.
By the assumption $m(n)>\ell(d)$, we conclude that $A_m$ does not act on the general fiber of $X'_k\rightarrow Z$ so 
it must act on $Z$.
Note that $Z$ is rationally connected, 
being the image of a rationally connected variety.
We take an $A_m$-equivariant resolution of singularities $Z'\rightarrow Z$. 
The variety $Z'$ is rationally connected and smooth, so $K_{Z'}$ is not pseudo-effective.
We run an $A_m$-equivariant MMP for $K_{Z'}$.
Proceeding inductively, we obtain a $d$-dimensional Fano variety 
that admits a $A_m$-action.
This contradicts the fact that $m(n)>\ell(d)$ for $d\leq n-1$. 
\end{proof}

\section{Examples and questions}
\label{sec:ex}

In this section, we consider several examples and questions related to the results of the article.

\begin{example}
\label{ex:optimal_example}
{\em 
Given a dimension $n$, let \(m \coloneqq c_{\mathrm{Fano}}(n) - n - 1 = \left\lceil \frac{1 + \sqrt{8n+9}}{2} \right\rceil - 1\).
Let $X$ be the following complete intersection in $\pp^{n+m}$:
\[
X \coloneqq \left\{\sum_{i=0}^{n+m} x_i = \sum_{i=0}^{n+m} x_i^2 = \cdots = \sum_{i=0}^{n+m} x_i^m = 0 \right\} \subset \pp^{n+m}.
\]
Then $X$ is a smooth Fano complete intersection of dimension $n$ by \Cref{lem:Fermat_ci_smooth}.  The symmetric group $S_{n+m+1}$ acts on $X$ by permutation of the variables and it's clear that this action is faithful.

This example is a maximally symmetric Fano weighted complete intersection for each dimension $n$ by \Cref{thm:max_symmetric_action} and has the largest possible index among such maximal examples by \Cref{thm:maximally_symmetric_largest_index}. We expect this to be a maximally symmetric Fano variety in every dimension.}

\end{example}

\begin{remark}
\label{rem:CY_sharp}
    \Cref{ex:optimal_example} gives examples in any dimension \(n\) showing that the bound for Fanos in \Cref{thm:max_symmetric_action} is sharp. One may ask whether the same can be done for Calabi--Yau complete intersections for \(n\geq 3\) (recall that by \Cref{prop:low_dim_wci} the bound \(c_{\mathrm{CY}}(n)\) is sharp for \(n=2\) but not for \(n=1\)). One obstacle to finding Calabi--Yau examples is that \Cref{lem:Fermat_ci_smooth} no longer holds if the degrees are not \((1,2,\ldots,m)\).

    For \(n=3\), the \((1,2,4)\)-Fermat complete intersection in \(\pp^6\) is smooth and therefore is a maximally symmetric Calabi--Yau weighted complete intersection.
    However, for \(n=4\), the degree \((1,2,5)\)-Fermat complete intersection in \(\pp^7\) is singular (and it even has non-isolated singularities). For \(n=4\), it turns out that the degree \((1,3,4)\)-Fermat complete intersection is smooth and thus exhibits a smooth maximally symmetric example, i.e. it achieves \(c_{\mathrm{CY}}(4)=8\). In general, the numerics to ensure smoothness seem complicated.

    Nevertheless, the upper bound \(c_{\mathrm{CY}}(n)\) is achieved for infinitely many values of $n$, namely whenever there happens to exist an $m$ such that the complete intersection with degrees $(1,2,\ldots,m)$ in $\mathbb{P}^{n+m}$ is Calabi--Yau.
\end{remark}

\begin{question}
\label{quest:CY_sharp}
    Is the bound in \Cref{thm:max_symmetric_action} for quasismooth Calabi--Yau weighted complete intersections sharp for all $n \geq 2$?
\end{question}

\begin{example}\label{exmp:index-1-Fano-Fermat}
{\em In any dimension \(n \), there exist maximally symmetric Fano--Fermat complete intersections of index \(1\). Indeed, this happens if \(X\) has degrees \((d_1,\ldots,d_m)\) with \(m=c_{\mathrm{Fano}}(n) - n - 1\) and \(n+m+1-\sum_{i=1}^m d_i=1\).
(For a concrete example, take \(d_i=i\) for all \(1\leq i\leq m-1\) and \(d_m=n+m-\frac{(m-1)m}{2}\).) However, as with \Cref{rem:CY_sharp}, the numerics to ensure smoothness seem complicated. E.g. for \(n=5\), the degree \((1,3,4)\)-Fermat--Fano in \(\pp^8\) is smooth, but the degree \((1,2,5)\)-Fermat--Fano is singular.

If \(X\) is smooth and if \(n\geq 210\), then \(X\) is birationally superrigid and in particular irrational by \cite[Theorem 1.2]{Zhuang-BSR-index-1}. Moreover, for \(n\geq 4\) any smooth such \(X\) is conjecturally birationally rigid and hence irrational \cite[Conjecture 5.1]{Pukhlikov-bir-rigid}.}
\end{example}

For \(n=3\), the \(X\) in \Cref{exmp:index-1-Fano-Fermat} is known as the symmetric sextic Fano threefold; it is a smooth Fano threefold with an intermediate Jacobian obstruction to rationality \cite{Beauville-symmetric-sextic}. (Moreover, any embedding of \(S_7\) into the birational automorphism group of a rationally connected threefold is conjugate to this action \cite[Proposition 1.1.(ii)]{Prokhorov_space_Cremona}.) For general \(n\), however, it is not clear how to guarantee smoothness in~\Cref{exmp:index-1-Fano-Fermat}.

\begin{question}
    Do there exist index \(1\) Fano--Fermat complete intersections as in~\Cref{exmp:index-1-Fano-Fermat} that are smooth for all \(n\)? In particular, do there exist irrational examples of maximally symmetric Fano--Fermat varieties for \(n\gg 0\)?
\end{question}

For rational varieties, \(S_{n+1}\) always acts on \(\mathbb P^n\) by permutation of coordinates, so we have an embedding \(S_{n+1}\leqslant\PGL_{n+1}(\C)\leqslant\Cr(n)\). In fact, one can get \(S_{n+2}\leqslant\PGL_{n+1}(\C)\):

\begin{example}\label{ex:symm-projective-rep}
{\em  
    For \(n\geq 1\), the projective representation \(S_{n+2}\to\PGL_{n+1}(\C)\) of degree \(n+1\) defines a faithful action of \(S_{n+2}\) on \(\pp^n\). \Cref{thm:toric} shows that this is the best one can do among toric varieties (apart from the $n = 2$ case).
}
\end{example}

There are also easy examples of rational \(n\)-dimensional Fanos with \(S_{n+3}\)-actions.

\begin{example}\label{ex:rat-n+3}
{\em 
    Let \(n\geq 1\) and define \[X \coloneqq \left\{\sum_{i=0}^{n+2} x_i = \sum_{i=0}^{n+2} x_i^2=0\right\} \subset \mathbb P^{n+2}.\] \(X\) is smooth by \Cref{lem:Fermat_ci_smooth} and it is isomorphic to a quadric; thus, it is a smooth rational \(n\)-dimensional Fano with a faithful \(S_{n+3}\)-action.
    }
\end{example}

Not all maximally symmetric Fano weighted complete intersections are complete intersections in projective space; that is, non-trivial finite covers can arise in \Cref{introthm:maximally-symmetric-Fano-WCI}.

\begin{example} 
{\em 
\label{ex:one_extra_weight}
Let $X \subset \pp^9(1^{(9)},2)$ be the following smooth weighted complete intersection, where the variables of the weighted projective space are $x_0, \ldots, x_8, y$:
$$X \coloneqq \left\{ \sum_{i = 0}^8 x_i = \sum_{i=0}^8 x_i^2 = \sum_{i=0}^8 x_i^3 = y^2 - \sum_{i=0}^8 x_i^4 = 0 \right\}.$$
Then $X$ is a Fano fivefold since $K_X \cong \mathcal{O}_X(-1)$ and it carries a faithful $S_9$-action by permutation of the $x_i$.  By \Cref{thm:max_symmetric_action}, $X$ is a maximally symmetric weighted complete intersection of dimension $5$.  It is a double cover of the Fermat $(1,2,3)$-complete intersection in $\pp^8$, which is the highest index maximally symmetric example by \Cref{thm:maximally_symmetric_largest_index}.
}
\end{example}

In \Cref{introthm:S8-4-fold-bounded}, we showed that the class of $S_8$-equivariant klt Fano $4$-folds is bounded. In fact, we are only aware of the following members in this class:
\begin{example}\label{ex:max-sym-Fano-4-fold}
{\em (Examples of Fano $4$-folds with $S_8$-actions.) Define following Fano--Fermat complete intersections in \(\pp^7\):
\begin{itemize}
    \item \(X_{123}\) of degrees \((1,2,3)\), and
    \item \(X_{124}\) of degrees \((1,2,4)\).
\end{itemize}
Then \(X_{123}\) and \(X_{124}\) are smooth \(S_8\)-equivariant Fano \(4\)-folds with \(\rho=1\). 
}
\end{example}

In contrast, the class of \(S_7\)-equivariant klt Fano \(4\)-folds is \textit{unbounded}:

\begin{example}\label{ex:S_7-Fano-4-fold}
{\em 
Let $X$ be the symmetric sextic Fano $3$-fold (\Cref{exmp:index-1-Fano-Fermat} for \(n=3\)). Then \(X\) is a smooth Fano variety of Picard rank one and admits faithful \(S_7\)-action.
The divisor class $-K_X$ is invariant under the action of $S_7$, and the divisor \(D \in |-4 K_X|\) defined by \( \sum_{i=0}^6 x_i^4\) is invariant under the \(S_7\)-action. For \(m\gg 0\), define
\[
Y_m \coloneqq \mathbb{P}(\mathcal{O}_X\oplus \mathcal{O}_X(-mD)) \xrightarrow{\pi_m} X.
\]
Note that $Y_m$ is endowed with a $S_7$-action.
Let $\mathcal{O}_{Y_m}(1)$ be the associated relative ample bundle. 
Let $E_m$ be the section corresponding to the exact sequence 
\[
1\rightarrow \mathcal{O}_X \rightarrow \mathcal{O}_X \oplus \mathcal{O}_X(-mD) \rightarrow \mathcal{O}_X(-mD) \rightarrow 1
\]
and let $F_m$ be the section corresponding the exact sequence 
\[
1\rightarrow \mathcal{O}_X(-mD) \rightarrow 
\mathcal{O}_X \oplus \mathcal{O}_X(-mD) 
\rightarrow \mathcal{O}_X \rightarrow 1. 
\]
The cokernel of $\pi_m^*(\mathcal{O}_X)\rightarrow \mathcal{O}_{Y_m}(1)$ is $\mathcal{O}_{E_m}(1)$, so the normal bundle of $E_m$ is $\mathcal{O}_{E_m}^\vee \otimes \mathcal{O}_{E_m}(-m\pi_m^*D)\simeq \mathcal{O}_{E_m}(4mK_{E_m})$. Analogously, the normal bundle of $F_m$ is $\mathcal{O}_{F_m}(-4mK_{F_m})$.
Note that 
\begin{equation}\label{eq:lin-equiv}
K_{Y_m}+E_m+F_m \sim \pi_m^*(K_X)
\end{equation}
and the divisor $\pi_m^*(-K_X)$ is nef.
We claim that $(Y_m,E_m)$ is log Fano.
Indeed, if $C$ is not contained in $F_m$,
then $(K_{Y_m}+E_m)\cdot C<0$ by the linear equivalence~\eqref{eq:lin-equiv}.
On the other hand, if $C$ is contained in $F_m$, then $-(K_{Y_m}+E_m)\cdot C =
F_m \cdot C + \pi^*_m(-K_X)\cdot C > 0$
by the normal bundle computation.
Therefore, for $\epsilon>0$ small enough, the pair $(Y_m,(1-\epsilon)E_m)$ is a klt Fano pair.
We conclude that $Y_m$ is a Mori dream space by~\cite[Corollary 1.3.2]{BCHM10}.
Thus, we may run an $S_7$-equivariant MMP for any $S_7$-invariant divisor on $Y_m$.
By the normal bundle computation, $E_m$ is covered by $E_m$-negative curves.
Since $E_m$ has Picard rank $1$, the $S_7$-equivariant $E_m$-MMP has a single step and contracts $E_m$ to a point. 
Let $\varphi_m\colon Y_m\rightarrow X_m$ be the $S_7$-equivariant contraction of $E_m$ to a point.
We obtain a variety $X_m$ of Picard rank one.
Note that $X_m$ is endowed with the action of $S_7$.
Next, we compute the mld of $X_m$.
To do this, let $C\subset E_m$ be a curve.
Then, the following sequence of equalities hold:
\[
\left(K_{Y_m}+E_m -\frac{1}{4m}E_m \right)\cdot C=
K_{E_m}\cdot C - \frac{1}{4m} E_m|_{E_m}\cdot C = K_{E_m}\cdot C - \frac{1}{4m} 4mK_{E_m}\cdot C=0.
\]
By the contraction theorem, we have that:
\[
\varphi_m^*(K_{X_m})=K_{Y_m}+\left(1-\frac{1}{4m}\right)E_m.
\]
We conclude that $X_m$ is an $S_7$-equivariant klt Fano variety with ${\rm mld}(X_m)=\frac{1}{4m}$. Indeed, the pair $(Y_m,E_m)$ is a log resolution of $X_m$. 
The minimal log discrepancies 
of a bounded set of projective varieties
can only take finitely many values.
We conclude that the varieties $X_m$ form
a sequence of unbounded
$S_7$-equivariant klt Fano $4$-folds.
}
\end{example}

In any dimension \(n\), the construction in \Cref{ex:S_7-Fano-4-fold} shows that, given a smooth \(S_k\)-equivariant \((n-1)\)-dimensional Fano variety of Picard rank one, one can construct an unbounded family of \(n\)-dimensional \(S_k\)-equivariant klt Fano varieties.

In the proof of \Cref{introthm:S8-4-fold-bounded}, to prove that $S_8$-equivariant Fano $4$-folds form a bounded family,
we use the following facts:
\begin{enumerate}
\item\label{item:boundedness-fact-1} the group $S_8$ does not act on Fano varieties of dimension at most $3$
nor Calabi--Yau varieties of dimension at most $2$ (see, e.g.,~\cite{BCDP-23}), 
\item\label{item:boundedness-fact-2} the group $S_8$ does not act smoothly on spheres of dimension at most $3$ (see~\cite{zim18}),
\item\label{item:boundedness-fact-3} the dual complex $\mathcal{D}(X,\Gamma)$ of a log Calabi--Yau pair of dimension at most $4$ is a quotient of a sphere of dimension at most $3$ (see~\cite[Proposition 5]{KX16}), and
\item\label{item:boundedness-fact-4} the boundedness of Fano $4$-folds with log discrepancies bounded away from zero (see~\cite[Theorem 1.1]{Bir21}).
\end{enumerate}
The statement~\eqref{item:boundedness-fact-3} is expected to hold in any dimension (see~\cite[Question 4]{KX16}).
On the other hand, if $m(n)$ is the largest integer for which 
$S_m$ acts faithfully on a $n$-dimensional Fano variety, 
then we expect that $S_m$ does not act on either a $\ell$-dimensional Fano variety with $\ell \leq n-1$.
Similarly, we expect that $S_m$ does not act on a $\ell$-dimensional Calabi--Yau variety with $\ell \leq n-1$.
Thus, we expect~\eqref{item:boundedness-fact-1} and~\eqref{item:boundedness-fact-2} to have analogous statements in higher dimensions.
Finally, \eqref{item:boundedness-fact-4} is known to hold in any dimension. 
This leads us to the following question:

\begin{question}\label{quest:max-symm-bound}
For \(n\geq 4\),
is the family of maximally symmetric $n$-dimensional Fano varieties bounded? 
\end{question}

Although we do not know whether Fano $4$-folds with $S_8$-actions are maximally symmetric, Theorem~\ref{introthm:S8-4-fold-bounded} implies that Fano $4$-folds endowed with an action of $S_k$, with $k\geq 8$, are bounded. Hence, the previous question has a positive answer in dimension $4$.
We do not have enough evidence for a positive answer of Question~\ref{quest:max-symm-bound} in higher dimensions. 
A better understanding of symmetric actions on Calabi--Yau varieties is needed to tackle this question.
The following question is very related to the boundedness one:

\begin{question}
Are maximally symmetric $n$-dimensional Fano varieties equivariantly exceptional?
I.e., are the quotients exceptional Fano varieties?
\end{question}

This holds in all the examples that we consider (\Cref{ex:max-sym-Fano-4-fold}). However, our tools do not allow to prove this statement in dimension $4$.
It would be interesting to find, if possible, singular examples of $S_8$-equivariant Fano $4$-folds. We do not know the existence of these:

\begin{question}\label{quest:sing-max-sym}
Are there examples of singular maximally symmetric Fano varieties?
What about singular $S_8$-equivariant Fano $4$-folds? 
\end{question}

In a similar vein, all the examples of maximally symmetric Fano varieties that we know are isolated. 
This motivates the following question:

\begin{question}\label{quest:moduli}
Do maximally symmetric Fano varieties have non-trivial moduli?
\end{question}

For \(n \geq 4\), the largest symmetric group that can embed into the Cremona group \(\Cr(n) = \Bir(\pp^n_\C)\) of rank \(n\) is not known.
\Cref{prop:ln-bound} gives a quadratic bound, which we do not expect to be sharp. \Cref{ex:rat-n+3} shows that \(S_{n+3}\) always embeds into \(\Cr(n)\). We also note that, for \(n\geq 3\), the integer \(c_{\mathrm{Fano}}(n)\) defined in \Cref{thm:max_symmetric_action}\eqref{part:max_symmetric_action-Fano} is always strictly greater than \(n+3\).
\begin{question}
    For \(n\geq 4\), is \(S_{n+3}\) the largest symmetric group that admits an embedding into \(\Cr(n)\)?  In particular, are all maximally symmetric Fano varieties irrational for $n \geq 3$?
\end{question}
Finally, we expect that~\Cref{introthm:quadratic-bound-klt}
can be improved by replacing $n^2$ with $n$.
However, this problem seems very challenging.
For instance, one would need to prove the analogous projective statement for Fano varieties.
However, we expect that weighted complete intersection singularities 
should be easier to deal with.
We propose the following question:
\begin{question}\label{quest:klt-WCI}
Let $S_{m(n)}$ be the largest symmetric group acting on an $n$-dimensional weighted complete intersection klt singularity. 
Do we have that $\lim_{n\rightarrow \infty} m(n)/n= 1$?
\end{question}
We expect that ideas similar to those of the proof of~\Cref{introthm:maximally-symmetric-Fano-WCI} lead to a positive answer for~\Cref{quest:klt-WCI}.
However, working in the local setting introduces extra difficulties, such as not having a well-defined degree of the equations that cut out the singularity.

\bibliographystyle{habbvr}
\bibliography{references}

\end{document}